\documentclass[11pt,a4]{article}

\usepackage{subfigure}
\usepackage{graphicx}
\usepackage{sidecap}
\usepackage{amsmath, amsthm, amssymb}
\usepackage{array}
\usepackage{url}
\usepackage{float}
\usepackage{enumitem}

\usepackage{psfrag}
\usepackage{color}
\usepackage{bm}

\usepackage[vcentering,dvips]{geometry}
\newtheorem{theorem}{Theorem}[section]
\newtheorem{lemma}[theorem]{Lemma}
\newtheorem{proposition}[theorem]{Proposition}
\newtheorem{corollary}[theorem]{Corollary}
\newtheorem{definition}[theorem]{Definition}
\newtheorem{example}[theorem]{Example}

\newtheorem{remark}[theorem]{Remark}

\newcommand{\E}{{\mathbb{E}} \hspace{.2mm}}
\renewcommand{\P}{{\mathbb{P}}}
\newcommand{\R}{{\mathbb{R}}}
\newcommand{\N}{{\mathbb{N}}}
\newcommand{\Z}{{\mathbb{Z}}}
\newcommand{\C}{{\mathbb{C}}}
\newcommand{\T}{{\mathbb{T}}}

 \newcommand{\scalprod}[2]{\left\langle #1,#2 \right\rangle}
\newcommand{\vertiii}[1]{{\left\vert\kern-0.25ex\left\vert\kern-0.25ex\left\vert #1 
    \right\vert\kern-0.25ex\right\vert\kern-0.25ex\right\vert}}
\newcommand{\Id}{\operatorname{Id}}    

\usepackage{enumitem}

\newcommand{\vct}[1]{\bm{#1}}
\newcommand{\mtx}[1]{\bm{#1}}

\newcommand{\A}{\mtx{A}}
\newcommand{\BB}{\mtx{B}}
\newcommand{\vv}{\vct{v}}
\newcommand{\x}{\vct{x}}
\newcommand{\X}{\vct{X}}
\newcommand{\veps}{\vct{\epsilon}}
\newcommand{\w}{\vct{\omega}}
\newcommand{\y}{\vct{y}}
\newcommand{\z}{\vct{z}}
\newcommand{\supp}{\operatorname{supp}}
\newcommand{\conv}{\operatorname{conv}}

\newcommand{\triple}{\hspace{.5mm} | \hspace{-0.5mm} | \hspace{-0.5mm}  | \hspace{.5mm}}

\begin{document}

\title{Interpolation via weighted $\ell_1$ minimization}

\author{Holger Rauhut\thanks{H.\ Rauhut is with RWTH Aachen University, Lehrstuhl C f{\"u}r Mathematik (Analysis), Templergraben 55, 52062 Aachen, Germany, \rm{rauhut@mathc.rwth-aachen.de}}, Rachel Ward \thanks{R.\ Ward is with the Mathematics Department at the University of Texas at Austin, 2515 Speedway, Austin, Texas, 78712,  \rm{rward@ma.utexas.edu}}}

\maketitle

\begin{abstract}
Functions of interest are often smooth and sparse in some sense, and both priors should be taken into account when interpolating sampled data.  Classical linear interpolation methods are effective under strong regularity assumptions, but cannot incorporate nonlinear sparsity structure.   At the same time, nonlinear methods such as $\ell_1$ minimization can reconstruct sparse functions from very few samples, but do not necessarily encourage  smoothness.  Here we show that \emph{weighted} $\ell_1$ minimization effectively merges the two approaches, promoting both sparsity and smoothness in reconstruction.  More precisely, we provide specific choices of weights in the $\ell_1$ objective to achieve approximation rates for functions with coefficient sequences in weighted $\ell_p$ spaces with $p \leq 1$.  We consider implications of these results for spherical harmonic and polynomial interpolation, in the univariate and multivariate setting.    Along the way, we extend concepts from compressive sensing such as the restricted isometry property and null space property to accommodate weighted sparse expansions; these developments should be of independent interest in the study of structured sparse approximations and continuous-time compressive sensing problems. 
\end{abstract}


\medskip

\noindent


\begin{center}
{\emph{Dedicated to the memory of Ted Odell}}
\end{center}

\section{Introduction}
This paper aims to merge classical smoothness-based methods for function interpolation with modern
sparsity constraints and nonlinear reconstruction methods.  We will focus on the classical interpolation problem, where given sampling points and associated function values,  we wish to find a suitably well-behaved function agreeing with the data up to some error. Classically, ``well-behaved" has been measured in terms of smoothness: the more derivatives a function has, the stronger the reconstruction rate obtained using linear methods.   More recently, areas such as compressive sensing have focused on sparsity rather than smoothness as a measure of complexity.  Results in compressive sensing imply that a function with sparse representation in a known basis can be reconstructed from a small number of suitably randomly distributed sampling points, using nonlinear techniques such as convex optimization or greedy methods.  In reality, functions of interest may only be somewhat smooth and somewhat sparse.  This is particularly apparent in high-dimensional problems, where sparse and low-degree tensor product expansions are often preferred according to the sparsity-of-effects principle, which states that most models are principally
governed by main effects and low order interactions.   In such situations, we might hope to synthesize smoothness and sparsity-based approaches.

Recall that the smoothness of a function tends to be reflected in the rapid decay of its Fourier series, and vice versa.  Smoothness can then be viewed as a structured sparsity constraint, with low-order Fourier basis functions being more likely to contribute to the best $s$-term approximation.  We will demonstrate that such structured sparse expansions are imposed by weighted $\ell_p$ coefficient spaces in the range $0 < p \leq 1$.  Accordingly, we will use weighted $\ell_1$ minimization, a convex surrogate for weighted $\ell_p$ minimization with $p < 1$, as our reconstruction method of choice.

\bigskip

\noindent The contributions of this paper are as follows.

\begin{enumerate}[leftmargin=10pt]

\item {\bf Function approximation.} We provide the first rigorous analysis for function interpolation using weighted $\ell_1$ minimization.  We show that with the appropriate choice of weights, one obtains approximation rates for functions with coefficient sequences lying in weighted $\ell_p$ spaces with $0 < p \leq 1$.  In the high-dimensional setting, our rates are better than those possible by classical linear interpolation methods, and require only mild smoothness assumptions.  For suitable choices of the weights, the number of sampling points required by weighted $\ell_1$ minimization to achieve a desired rate grows only linearly or even just logarithmically with the ambient dimension, rather than exponentially. We illustrate the improvement of weighted $\ell_1$ minimization compared to unweighted $\ell_1$ minimization on several specific examples, including spherical harmonic interpolation and tensorized Chebyshev and Legendre polynomial interpolation. We expect that the results for polynomial interpolation should have applications in uncertainty quantification \cite{uq} and, in particular, in the computation of generalized {polynomial chaos} expansions \cite{cds1,cds2}.

\item {\bf Weighted sparsity.} In order to derive stable and robust recovery guarantees for weighted $\ell_1$ minimization, we generalize the notion of restricted isometry property in compressive sensing to the weighted restricted isometry property, and also develop notions of weighted sparsity which take into account prior information on the likelihood that any particular index contributes to the sparse expansion.  These developments should be of independent interest as tools that can be used more generally for the analysis of structured sparsity models and continuous-time sparse approximation problems.  

\item {\bf Structured random matrices for compressive sensing.}  
It is by now well established \cite{cata06, ruve08, ra10} that an $s$-sparse trigonometric polynomial of maximal degree $N$ can be recovered efficiently from its values at $m \asymp s \log^4(N)$ sampling points drawn independently from the uniform measure on their domain. 
More generally, such near-optimal sparse recovery results hold whenever the underlying sparsity basis is uniformly bounded like the trigonometric system, so as to be incoherent with point samples \cite{ra10}.   However, these results become weak if the $L_{\infty}$ norms of the functions in the orthonormal system grow with the maximal degree $N$.  If this growth is not too sharp, one may apply a preconditioning technique proposed in  \cite{rw10} to transform the system into a uniformly bounded one.  However, often these preconditioning techniques are not enough to make the basis functions uniformly bounded.  Here, we show that as long as lower-degree polynomials are preferred in the sparse expansions under consideration, one may still derive sparse recovery guarantees using weighted $\ell_1$ minimization with weights which grow at least as quickly as the $L_{\infty}$ norms of the functions they correspond to.

\end{enumerate}

\noindent We will assume throughout that the sampling points for interpolation are drawn from {a suitable probability distribution}, and we focus only on the setting where interpolation points are chosen in advance, independent of the target function.  

\subsection{Organization} The remainder of the paper is organized as follows.
In Sections 1.2 and 1.3, we introduce weighted $\ell_p$ spaces and 
{discuss} how they promote smoothness and sparsity. 
In Section 1.4 we state two of the main interpolation results, and in Section 1.5 we introduce the concept of weighted restricted isometry property for a linear map. Section 1.6 discusses previous work on weighted $\ell_1$ minimization, and in Section 1.7 we compare our main results to those possible with linear reconstruction methods. 
 In Section 2 we discuss the implications of our main results for spherical harmonic and tensorized polynomial bases, and provide a numerical illustration.  We further analyze concepts pertaining to weighted sparsity in Section 3, and in Section 4 we elaborate on the weighted restricted isometry property and weighted null space property.  In Section 5 we show that matrices arising from orthonormal systems have the weighted restricted isometry property as long as the weights are matched to the $L_{\infty}$ norms of the function system, and we finish in Section 6 by presenting our main results on interpolation via weighted $\ell_1$ minimization. 
 
\subsection{Weighted sparsity}

We will work with coefficient sequences $\x$ indexed by a set $\Lambda$ 
{which} may be finite or countably infinite.  We will associate to a vector $\w = (\omega_{j})_{j \in \Lambda}$ of weights $\omega_j \geq 1$ the weighted $\ell_p$ spaces
\begin{eqnarray}
\label{weighted_norm}
 \ell_{\omega,p} := \left\{ \x = (x_j)_{j \in \Lambda},  \quad  \|\x\|_{\omega, p} := \Big( \sum_{j \in \Lambda} \omega^{2-p}_{j} |x_{j}|^p \Big)^{1/p} < \infty \right\}, \quad 0 < p \leq 2.
\end{eqnarray}
Also central to our analysis will be the weighted $\ell_0$-``norm'', $$\| \x \|_{\omega, 0} = \sum_{\{j: x_{j} \neq 0\}} \omega^2_{j},$$
which counts the squared weights of the non-zero entries of $\x$.  We also define {the} weighted cardinality of a set $S$ to be 
$\omega(S) := \sum_{j \in S} \omega^2_{j}.$   Since $\omega_{j} \geq 1$ by assumption, we always have $\omega(S) \geq |S|$, the cardinality of $S$.
When $\w \equiv 1$, these weighted norms reproduce the standard $\ell_p$ norms, in which case we use the standard $\ell_p$-notation $\| \cdot \|_{p}$.  The exponent $2-p$ in the definition of the spaces $\ell_{\omega,p}$ is somewhat uncommon but turns out to be the most convenient definition for our purposes.  For instance, the exponent must scale this way in order to apply a Cauchy-Schwarz inequality to weighted $s$-sparse vectors:
$$
\| \x  \|_{\omega,1} = \sum_{j \in \Lambda}\omega_j |x_j|  \leq \sqrt{\sum_{j \in \Lambda}\omega_j^2} \sqrt{\sum_{j \in \Lambda} |x_j|^2}   = \sqrt{\omega(\Lambda)} \| \x \|_2 \leq \sqrt{s} \| \x\|_2, \quad \forall \x: \| \x \|_{\omega,0} \leq s
$$
Consequently, error bounds we present will be in terms of the $\ell_{\omega,p}$ norms scaled as such; for $p=1$ we have weighted $\ell_1$ error bounds, for $p=2$ we have unweighted $\ell_2$ error bounds.

For $\x \in \ell_{\omega,p}$ and for a subset $S$ of the index set $\Lambda$, we define $\x_S \in \ell_{\omega,p}$ as the restriction of $\x$ to $S$.   
For $s \geq 1$, the error of best weighted $s$-term approximation
of the vector $\x \in \ell_{\omega,p}$ is defined as
\begin{equation}
\label{weightedapprox}
\sigma_s(\x)_{\omega, p} = \inf_{\z : \|\z\|_{\omega, 0} \leq s} \|\x-\z\|_{\omega, p}.
\end{equation}
Unlike unweighted best $s$-term approximations, weighted approximations of vectors are not straightforward 
to compute in general.  Nevertheless, we will show in Section~\ref{stechkin} how to approximate $\sigma_s(\x)_{\omega,p}$ using a quantity that can easily computed from $\x$ by sorting and thresholding.

\subsection{Weighted $\ell_p$ spaces for smoothness and sparsity} 
The weighted $\ell_p$ coefficient spaces introduced in the previous section can be used to define weighted function spaces.
For a bounded domain ${\cal D}$, let $\psi_j : {\cal D} \to \C${, $j \in \Lambda$,} be a sequence of functions indexed by the set $\Lambda$ which are orthonormal with respect to the probability measure $\nu$ on ${\cal D},$ that is, 
\[
\int_{\cal D} \psi_j(t) \overline{\psi_k(t)} d\nu(t) = \delta_{j,k} \quad \mbox{ for all } j,k.
\]
We will call 
{$\nu$} the orthogonalization measure associated to the system $( \psi_j )_{j \in \Lambda}$.  
The function spaces we consider are the weighted quasi-normed spaces,
\[
  \label{our_spaces}
  S_{\omega,p} := \left\{ f(t) = \sum_{j \in \Lambda} x_j \psi_j(t), \hspace{2mm}  t \in {\cal D}, \hspace{2mm} \triple f \triple_{\omega,p} := \| \vct{x} \|_{\omega,p}  < \infty \right\}, \quad 0 < p \leq 1,
\]
again with $\omega_j \geq 1$ implicitly assumed.  The best $s$-term approximation to $f \in S_{\omega,p}$  is the function
\begin{equation}
\label{bestapprox}
f_{S} = \sum_{j \in S} x_j \psi_j,
\end{equation}
where $S \subset \Lambda$ is the set realizing the weighted best $s$-term approximation of $\x$, and the best weighted $s$-term approximation error is
\begin{equation}
\label{weightedapproxF}
\sigma_s(f)_{\omega, p} = \sigma_s(\x)_{\omega, p}.
\end{equation}
The following Stechkin-type estimate, described in more detail in Section \ref{stechkin}, can be used to bound the best $s$-term approximation of a vector by an appropriate weighted vector norm:
\begin{equation}
\label{stech1}
\sigma_{s}(\x)_{\omega,q} \leq \big( s - \|\w\|_\infty^2 \big)^{1/q-1/p} \|\x\|_{\omega,p}, \quad p < q \leq 2, \quad \| \w \|_{\infty}^2 < s.
\end{equation}
This estimate illustrates how a small $\ell_p-$norm for $p < 1$ supports small sparse approximation error.   Conditions of the form $\| \w \|^2_\infty < s$ are somewhat natural in the context of weighted sparse approximations, as those indices with weights $\omega_{j}^2 > s$ cannot possibly contribute to best weighted $s$-term approximations.  
This means that we can 
usually replace a countably-infinite set $\Lambda$ by the finite subset $\Lambda_0 \subset \Lambda$ corresponding to indices with 
weights $\omega_j^2 < s$ {(or, for technical reasons, $\omega_j^2 \leq s/2$)}, if such a finite set exists.  

\subsection{Interpolation via weighted $\ell_1$ minimization}

In treating the interpolation problem, we first assume that the index set $\Lambda$ is finite with $N = | \Lambda |$.
Given sampling points $t_1,\hdots,t_m \in {\cal D}$ and $f= \sum_{j \in \Lambda} x_j \psi_j$ 
we can write the vector of sample values $\y = (f(t_\ell))_{\ell=1,\hdots,m}$ succinctly in matrix form as $\y = \A \x$,
where $\A$ is the \emph{sampling matrix} with entries
\[
\label{sample_mat}
A_{\ell,j} = \psi_j(t_\ell), \quad \ell=1,\hdots,m, \quad j \in \Lambda. 
\]
Better sets of interpolation points are usually associated with better condition number for the sampling matrix $\A$.   In our theorems, the sampling points are drawn independently from the orthogonalization measure $\nu$ associated to the orthonormal system $(\psi_j)$; as a consequence, the random matrix $\A \A^*$, properly normalized, is the identity matrix in expectation.

\vspace{3mm}
  
\noindent We will consider the setting where the number of samples $m$ is smaller than the ambient dimension $N$, in which case there are infinitely many functions $g \in S_{\omega,p}$ which interpolate the given data.  From within this infinite set, we would like to pick out 
{the} function of minimal quasi-norm $\triple g \triple_{\omega,p}$.  However, this minimization problem only becomes tractable once  $p=1$ whence the quasi-norm becomes a norm.   As a convex relaxation to the weighted quasi-norm $p < 1$, we consider for interpolation the function $f^{\sharp}(t) = \sum_{j \in \Lambda} x_j^{\sharp} \psi_j(t)$ whose coefficient vector $\x^{\sharp}$ is the solution of the weighted $\ell_1$ minimization program
\[
\label{eq:wl1}
\min \| \z \|_{\omega,1} \mbox{ subject to } \A \z = \y
\]
The equality constraint in the $\ell_1$ minimization ensures that the function $f^\sharp$ interpolates
$f$ at the points $t_\ell$, that is, $f^\sharp(t_\ell) = f(t_\ell)$, $\ell=1,\hdots,m$. 
Let us give the following result on interpolation via weighted $\ell_1$ minimization with respect to $\| \cdot \|_{\omega,1}$.

\begin{theorem}\label{thm:intro1}
Suppose $(\psi_j)_{j \in \Lambda}$ is an orthonormal system of finite size $| \Lambda | = N$, and consider weights $\omega_j \geq \|\psi_j\|_\infty$. 
{For $s \geq 2 \|  \w \|_{\infty}^2$},
fix a number of samples 
\begin{equation}\label{bound:RIP:m}
m \geq c_0 s \log^3(s) \log(N),
\end{equation}
and suppose that $t_\ell$, $\ell=1,\hdots,m,$ are sampling points drawn i.i.d. from the orthogonalization measure associated to $(\psi_j)_{j \in \Lambda}$.  With probability exceeding $1 - N^{-\log^3(s)}$, the following holds for all 
functions $f  = \sum_{j \in \Lambda} x_j \psi_j$: given samples $y_\ell = f(t_\ell)$, $\ell=1,\hdots,m$,  let $\x^\sharp$ be the solution of
\[
\min \|\z\|_{\omega,1} \mbox{ subject to } \A \z = \y
\]
and set $f^\sharp(t) = \sum_{j\in \Lambda} x_j^\sharp \psi_j(t)$. Then the following error rates are satisfied: 
\begin{align}
\| f - f^\sharp\|_{L_\infty}  \leq \triple f - f^\sharp \triple_{\omega,1}  & \leq  c_1 \sigma_s(f)_{\omega,1},\notag\\
\| f - f^\sharp \|_{L_2} & \leq d_1 \sigma_s(f)_{\omega,1} / \sqrt{s}. \notag
\end{align}

\noindent Here $\sigma_s(f)_{\omega,1}$ is the best $s$-term approximation error of $f$ defined in \eqref{weightedapproxF}.
The constants $c_0$,$c_1,$ and $d_1$ are universal, independent of everything else.
\end{theorem}
This interpolation theorem is nonstandard in two respects: the number of samples $m$ required to achieve a prescribed rate scales only logarithmically with the size of the system, and the error guarantees are given by best $s$-term approximations in weighted coefficient norms.  

The constraint on the weights $\omega_j \geq \| \psi_j \|_{\infty}$ allows us to bound the $L_{\infty}$ norm by the weighted $\ell_1$ coefficient norm: for a function $f \in S_{\omega,p}$,
$$\| f \|_{L_{\infty}} =\sup_{t \in {\cal D}}  \left| \sum_{n=-\infty}^{\infty} x_n \psi_n(t)  \right|  \leq  \sup_{t \in {\cal D}} \sum_{n=-\infty}^{\infty} | x_n | |\psi_n(t) |  \leq  \sum_{n=-\infty}^{\infty} | x_n | \omega_n =  \triple f \triple_{\omega,1},$$
and so if $f_0 = \sum_{j \in S} x_j \psi_j$ with $|S| = s$ is the best $s$-term approximation to $f$ in the $L_{\infty}$ norm, 
then by the Stechkin-type estimate \eqref{stech1} 
with $q=1$ we have
$$
\| f - f_0 \|_{L_\infty} \leq \triple f - f_0 \triple_1 \leq  (s- \| \w \|_{\infty}^2)^{1-1/p} \triple f \triple_{\omega,p}, \quad \quad p < 1.
$$
By choosing weights so that $\omega_j \geq \| \psi_j \|_{L_{\infty}} + \| \psi_j' \|_{L_{\infty}}$, one may also arrive at bounds of the form $\| f \|_{L_\infty} + \| f' \|_{L_{\infty}}  \leq \triple f  \triple_1$, reflecting how steeper weights encourage more smoothness.  We do not pursue such a direction in this paper, but this may be interesting for future research.

In Section \ref{alltogether} we will prove a more general version of Theorem \ref{thm:intro1} showing robustness of weighted $\ell_1$ minimization to noisy samples, 
$\y = \A \x + \vct{\xi}$.
Using this robustness to noise, we will be able to treat the case where the index set $\Lambda$ is countably infinite, by regarding the values $f(t_\ell), \ell=1,\dots,m,$ as noisy samples of a finite-dimensional approximation to $f$.  For example, this will allow us to show the following result.

\begin{theorem}\label{thm:infRIPprob}
Suppose $(\psi_j)_{j\in \Lambda}$ is an orthonormal system, consider weights $\omega_j \geq \|\psi_j\|_\infty$, and for a parameter $s \geq 1$, let $N = | \Lambda_0 |$ where $\Lambda_0 = \{ j : \omega_j^2 \leq s/2\}$.  Consider a number of samples 
$$
m \geq c_0 s \log^3(s) \log(N). 
$$
Consider a fixed function $f  = \sum_{j \in \Lambda} x_j \psi_j$ with $\triple f \triple_{\omega,1} < \infty$.  Draw sampling points $t_\ell$, $\ell=1,\hdots,m,$ independently from the orthogonalization measure associated to $(\psi_j)_{j\in \Lambda}$.  Let  $\A \in \C^{m \times N}$ be the sampling matrix with entries $A_{\ell,j} = \psi_j(t_\ell)$.  Let $\eta > 0$ and $\varepsilon \geq 0$ be such that 
{$\eta \leq \triple f - f_{\Lambda_0} \triple_{\omega,1} \leq \eta (1+\varepsilon)$}.
From samples $y_\ell = f(t_\ell)$, $\ell=1,\hdots,m$, let $\x^\sharp$ be the solution of
\[
\min \|\z\|_{\omega,1} \mbox{ subject to }  \| \A \z - \y \|_2 \leq \left(\frac{m}{s} \right)^{1/2} \eta
\]
and set $f^\sharp(t) = \sum_{j\in \Lambda_0} x_j^\sharp \psi_j(t)$. Then with probability exceeding $1 - N^{-\log^3(s)}$, 
\begin{align}
\| f - f^\sharp  \|_{L_{\infty}} \leq \triple f - f^\sharp \triple_{\omega,1} & \leq c_1 \sigma_s(f)_{\omega,1}, \notag\\
\|f-f^\sharp\|_{L^2}  & \leq d_1\sigma_s(f)_{\omega,1} / \sqrt{s}.\notag
\end{align}
\noindent Above, $c_0$ is an absolute constant and $c_1, d_1$ are constants which depend only on the distortion $\varepsilon$.
\end{theorem}
Several remarks should be made about Theorem \ref{thm:infRIPprob}.
\begin{enumerate}
\item The minimization problem in Theorem \ref{thm:infRIPprob} requires knowledge of, or at least an estimate of, the tail bound $\triple f - f_{\Lambda_0} \triple_{\omega,1}$.  It is possible to avoid this using greedy or iterative methods; see \cite{jo2013iterative} for one such method, a weighted version of the iterative hard thresholding algorithm \cite{blda09}.  In subsequent corollaries of this result, we will assume exact knowledge of the tail bound, $\eta = \triple f - f_{\Lambda_0} \triple_{\omega,1},$ for simplicity of presentation.
\item If the size $N$ of $\Lambda_0$ is polynomial in $s$, then the number of samples reduces to $m \geq C s \log^4(s)$ to achieve reconstruction with probability $> 1 - s^{-\log^3(s)}$.  
\end{enumerate}

\subsection{Weighted restricted isometry property}

One of the main tools we use in the proofs of Theorems \ref{thm:intro1} and \ref{thm:infRIPprob} is the \emph{weighted restricted isometry property} ($\w$-RIP) for a linear map $\A: \C^N \rightarrow \C^m$, which generalizes the concept of restricted isometry property in compressive sensing.

\begin{definition}[$\w$-RIP constants]
\label{def:weighted:RIP}
For $\A \in \C^{m \times N}$, $s \geq 1$, and weight $\omega$, the $\w$-RIP constant $\delta_{\omega,s}$ 
associated to 
$\A$ is the smallest number for which
\begin{equation}\label{def:wRIP}
(1-\delta_{\omega,s}) \|\x\|_{2}^2 \leq \| \A \x \|_2^2 \leq (1+\delta_{\omega,s}) \| \x \|_2^2
\end{equation}
for all $\x \in \C^N$ with $\| \x \|_{\omega,0} = \sum_{j \in \text{supp}(\x)} \omega_j^2 \leq s$.
\end{definition}
For weights $\w \equiv 1$, the $\w$-RIP reduces to the standard RIP,  as introduced in \cite{cata06, ct05}.
For general weights $\omega_j \geq 1$, {the} $\w$-RIP is a weaker assumption for a matrix than the standard RIP, as it requires the map to act as a near-isometry on a smaller set.

\begin{example} \emph{
Consider weights of the general form $\omega_j = j^{\alpha/2}$ with $\alpha > 0$.   We may then take $N = s^{1/\alpha}$, as even single indices $j > N$ have weighted cardinality exceeding $s$.  Observe that if $\| \x \|_{\omega,0} \leq s,$ then $\x$ is supported on an index set of (unweighted) cardinality at most $\alpha^{1/\alpha} s^{1/(\alpha+1)}$.  Following the approach of \cite{bddw08}, see also \cite{Baraniuk:2010, bd09}, taking a union bound and applying covering arguments, one may argue that an $m \times N$  i.i.d.\ subgaussian random matrix has the $\w$-RIP with high probability once
$$
m =\mathcal{O}\left( \alpha^{(1/\alpha)-1}s^{1/(\alpha+1)} \log{s}\right).
$$
This is a smaller number of measurements than the $m =\mathcal{O}\big( s \log(N/s) \big)$ lower bound required for the unweighted RIP.  This observation should be of independent interest, but we focus in this paper on random matrices formed by sampling orthonormal systems.}
\end{example}

\subsection{Related work on weighted $\ell_1$ minimization}
Weighted $\ell_1$ minimization has been analyzed previously in the compressive sensing literature.  Weighted $\ell_1$ minimization with weights $\omega_j \in \{0,1\}$ was introduced independently in the papers \cite{amin2009weighted, bmp07, vl09} and extended further in \cite{j10}. The paper \cite{fmsy12} seems to be the first to provide conditions under which weighted $\ell_1$ minimization is stable and robust under weaker sufficient conditions than the analogous conditions for standard $\ell_1$ minimization for general weights.  Improved sufficient conditions were recently provided for this setting in \cite{yb13}.  Recovery guarantees were also provided in \cite{peng2014weighted}, also with a focus on function interpolation with applications to polynomial chaos expansions.  Still, the analysis in all of these works is based on the standard restricted isometry property and this limits the extent to which their recovery guarantees improve on those for unweighted $\ell_1$ minimization.

Weighted $\ell_1$ minimization has also been considered under probabilistic models.  In \cite{xu10}, the vector indices are partitioned into two sets, and indices on each set have different probabilities $p_1$, $p_2$ of being nonzero; the weights are constant on each of the two sets. 
The papers \cite{amin2009weighted, khajehnejad2010analyzing} provide further analysis in this setting where the entries of the unknown vector fall into two or more sets, each with a different probability of being nonzero. Finally, the paper \cite{mp13} considers a full Bayesian model, where certain probabilities are associated with each component of the signal in such a way that the probabilities vary in a ``continuous" manner across the indices.  All of these works take a Grassmann angle approach, and the analysis is thus restricted to the setting of Gaussian matrices and to noiseless measurements.

\subsection{Comparison with classical interpolation results}

Although weighted $\ell_1$ minimization was recently investigated empirically in \cite{fk12,rasc14,peng2014weighted} for multivariate polynomial interpolation in the context of polynomial chaos expansions, weighted $\ell_p$ spaces, for $0 < p \leq 1$, are nonstandard in the interpolation literature.  More standard spaces 
are the weighted $\ell_2$ spaces (see e.g. \cite{kunis2007stability, nsw02}) 
 \begin{equation}
  \label{ssdad}
  S_{\alpha} := \left\{ f = \sum_{j \in \Z^d} x_j \phi_j, \quad \| f \|_{\alpha}^2 := \sum_{j \in \Z^d} \alpha_j | x_j |^2  < \infty \right\}.
  \end{equation}
  where $(\phi_j)$ is the tensorized Fourier basis on the torus $\T^d$.  Note that here we refer to weighted $\ell_2$ spaces, as opposed to the weighted norm introduced in \eqref{weighted_norm} which reduces to the unweighted $\ell_2$ norm when $p=2$.
  
  For the choice of weights $\alpha_j = (1 + \| j \|_2^2)^{r}$, $j \in \Z^d$, on these spaces coincide with the Sobolev spaces $W^{r,2}(\T^d)$ 
  of functions with $r$ derivatives in $L_2(\T^d)$. 
Optimal interpolation rates for these Sobolev spaces are obtained using smooth and localized kernels (as opposed to polynomials).  For example, from equispaced points on the $d$-dimensional torus with mesh size $h > 0$, \cite{nsw02} derives error estimates of the form
 \[
 \label{sobolev_bound}
 \| f - f^{\#} \|_{\infty} = \mathcal{O}(h^{r - d/2}) \| f \|_{\alpha}.
 \]
Writing out this error rate in terms of the number of samples $m = (1/h)^d$, one obtains
$$
\| f - f^{\#} \|_{\infty} \leq \mathcal{O}(m^{1/2-r/d})\| f \|_{\alpha}.
$$
The dependence of the exponent $1/2-r/d$ in $d$ means that we face the curse of dimension.

Such behavior in $d$ may be alleviated for linear interpolation by passing to Sobolev spaces of mixed smoothness, which are endowed with the
norm
\[
 \| f \|_{r,\operatorname{mix}}^2 := \sum_{j \in \Lambda} | x_j |^2 \prod_{\ell=1}^d (1+|j_\ell|^2)^{r}, \quad \mbox{ where } f = \sum_{j \in \Z^d} x_j \phi_j.
\]
Sampling on a sparse grid with $m$ sampling points and reconstructing via Smolyak's algorithm leads to the error bound
\[
\| f - f^{\#} \|_{\infty} = \mathcal{O}(m^{1/2-r} \log(m)^{(d-1)r}) \|f \|_{r,\operatorname{mix}}, 
\]
see \cite{te93-1} and \cite[Theorem 6.11]{bydusiul14}. A matching lower bound has recently been proved in \cite{cokusi14}. This means that the dependence in $d$ can be avoided at the term $m^{1/2-r}$.
However, the additional logarithmic term still exhibits exponential scaling in $d$.

In contrast, Theorem \ref{thm:intro2} implies that weighted $\ell_1$ minimization gives the rate 
$$
\| f - f^{\#} \|_{\infty} \leq \mathcal{O}\left( \frac{m}{\log^3(m)\log(N)} \right)^{1-1/p} \triple f \triple_{\omega,p} \mbox{ where } 0 < p < 1.
$$
Here, the dependence in $d$ is hidden in $N = \#\Lambda_0 = \# \{ j \in \Z^d : \omega_j^2 \leq s/2\}$ 
which in turn depends on the weight $\omega$. We discuss two reasonable choices which correspond
to the ones leading to Sobolev and mixed Sobolev spaces in the $\ell_2$-case described above. Let first
\[
\omega_j = (1+\|j\|_2)^r, \quad j \in \Z^d.
\]
Then 
\[
N = \# \{ j \in \Z^d : (1+\|j\|_2)^{2r} \leq s/2 \}  \quad \mbox{ so that } \log(N) \leq C \frac{d}{r} \log(s).
\]
This means that the approximation error decays like
$$
\| f - f^{\#} \|_{\infty} \leq C \left( \frac{r\,m}{d \log^4(m)} \right)^{1-1/p} \triple f \triple_{\omega,p}.
$$
Hence, the approximation error depends only polynomially on the dimension $d$, and we avoided the curse of dimension by
working with the stronger norm  $\triple f \triple_{\omega,p}$ and using nonlinear reconstruction.

We can further reduce the dependence in $d$ by using the weights
\[
\omega_j = \prod_{\ell=1}^d (1+|j_\ell|)^{r}, \quad j \in \Z_d.
\] 
In fact, the set $\Lambda_0 = \{ j \in \Z^d : \omega_j^2 \leq s/2\} = \{ j \in \Z^d : \prod_{\ell=1}^d (1+|j_\ell|) \leq (s/2)^{1/(2r)} \}$ 
is a hyperbolic cross. Its size can be bounded as 
\[
N = \# \Gamma_0^s \leq e^2 \left((s/2)^{1/(2r)}\right)^{2+\log_2(d)}
\]
see \cite[Proof of Theorem 4.9]{KSU13}. Therefore, $\log(N) \leq C r^{-1} \log(d) \log(s)$ and the error bound resulting from Theorem~\ref{thm:intro2} is
\[
\| f - f^{\#} \|_{\infty} \leq C \left( \frac{r\, m}{\log(d) \log^4(m)} \right)^{1-1/p}.
\]
The dependence in the dimension $d$ is only logarithmic for this choice of weight function. Hence, passing from (mixed) Sobolev spaces
to weighted $\ell_p$-spaces with $p < 1$ on the Fourier coefficients, and from linear interpolation to nonlinear reconstruction, 
may lead to a significant improvement in the error rates and may avoid the curse of dimension.

\section{Examples}

In this section we consider several examples and demonstrate how Theorem \ref{thm:infRIPprob} gives rise to various sampling theorems for polynomial and spherical harmonic interpolation.  One could derive similar results in weighted $\ell_p$ spaces using Theorem \ref{thm:intro2}. 

{\subsection{Spherical harmonic interpolation}\label{sec:sphere}}

\begin{figure}
\quad \quad \quad \quad {\includegraphics[width=2.8cm, height=3cm]{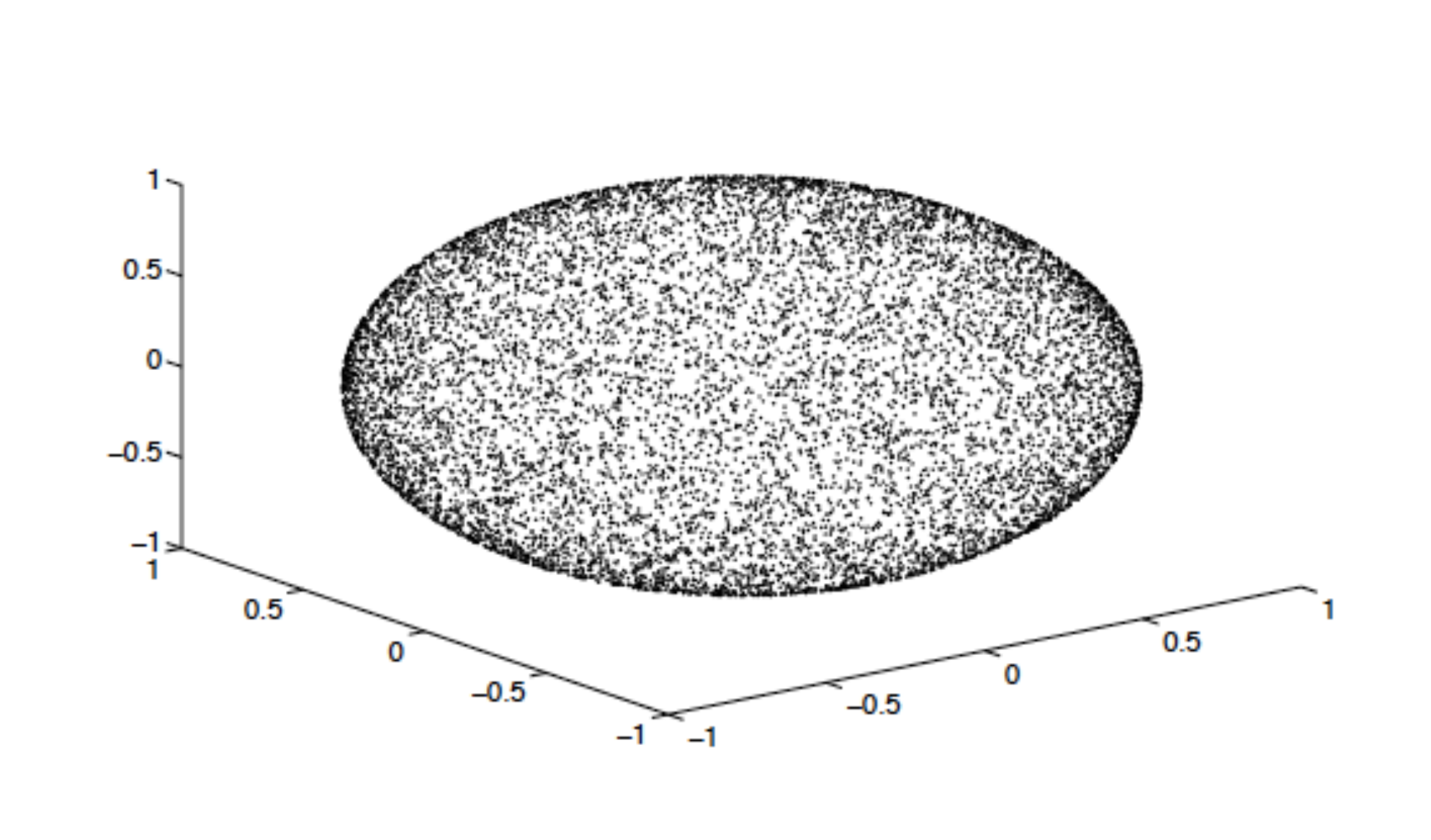}} \quad  \quad \quad 
{\includegraphics[width=2.8cm, height=3cm]{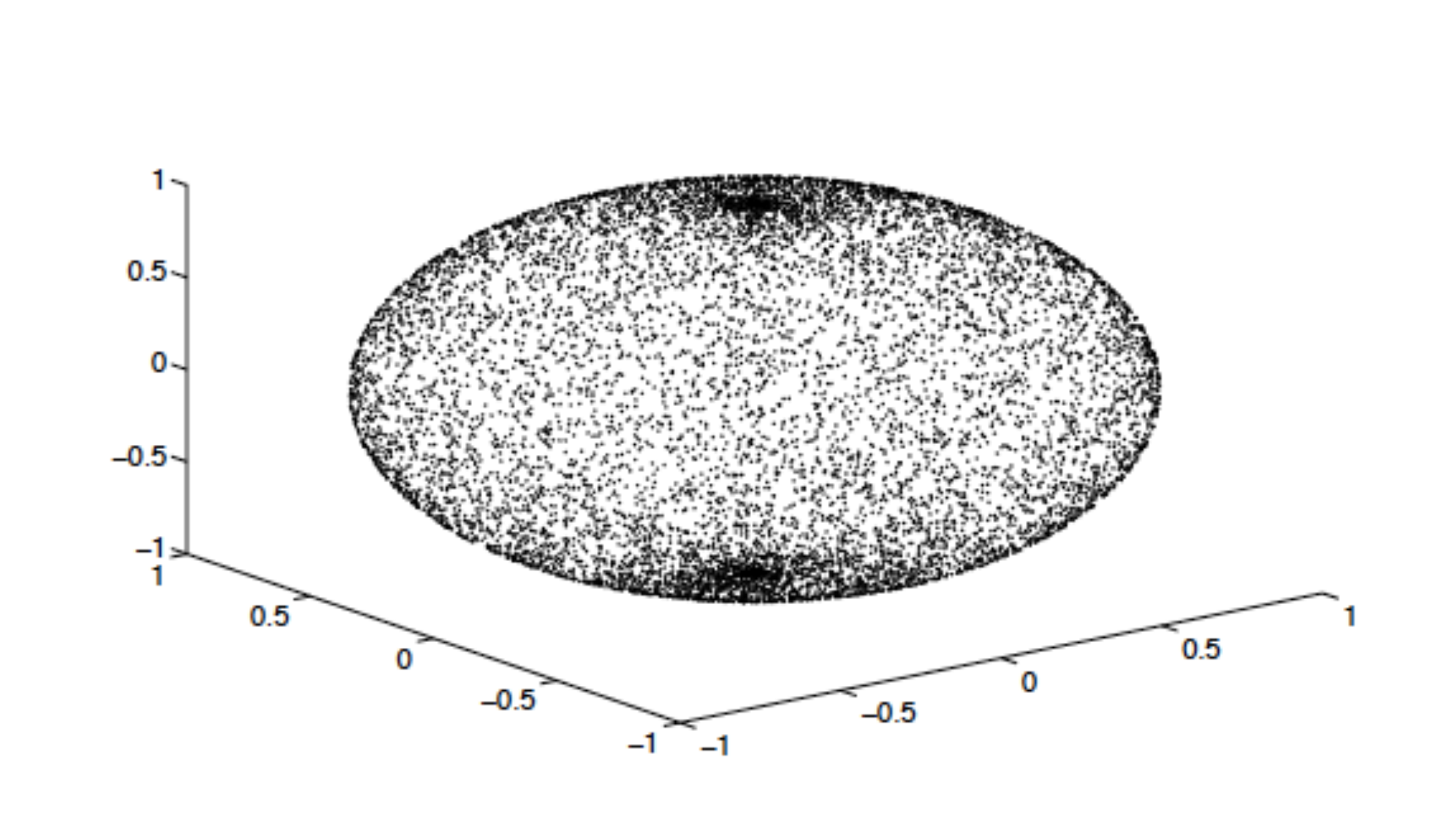}} \quad  \quad \quad 
{\includegraphics[width=2.8cm, height=3cm]{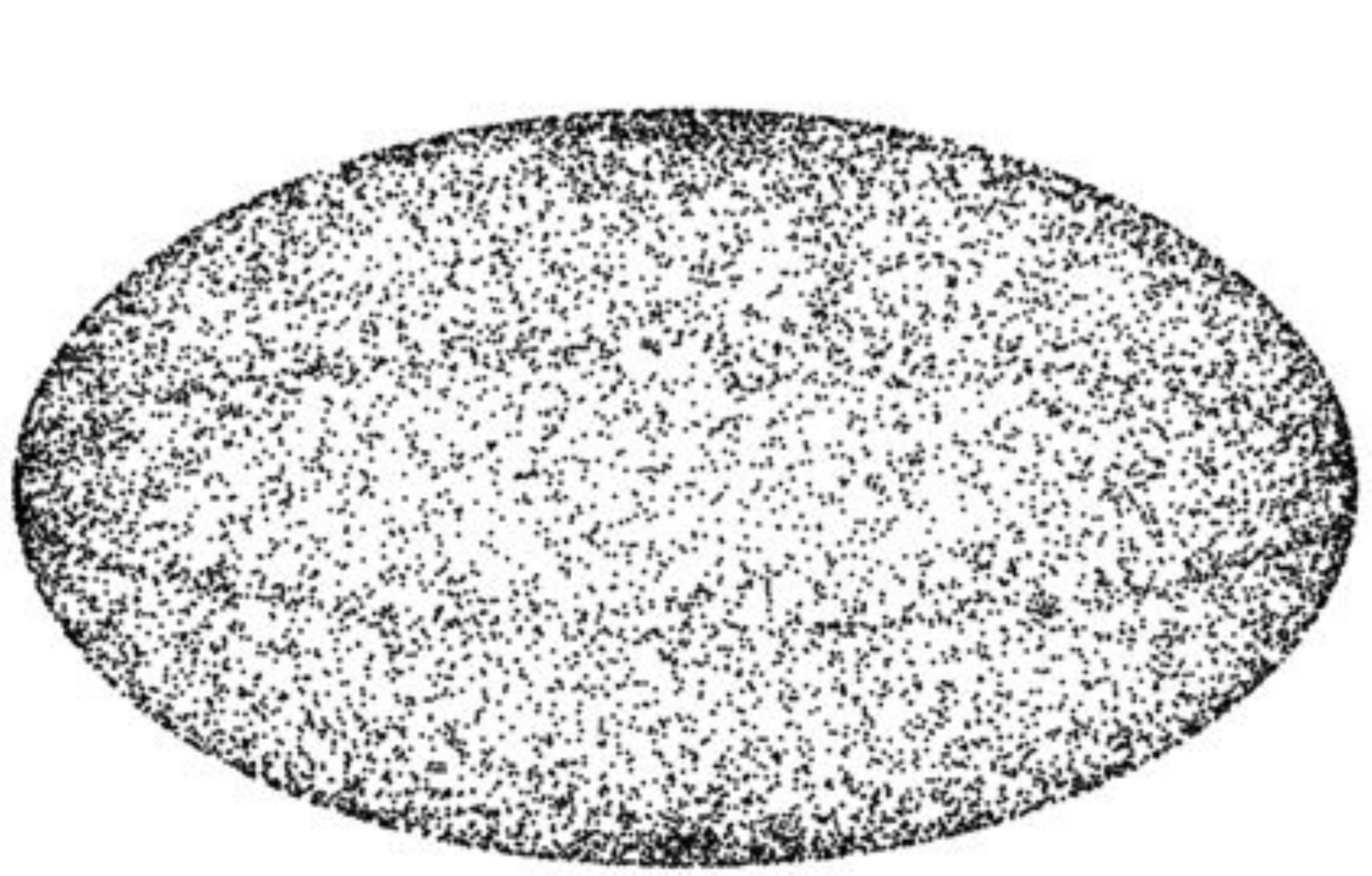}} \\
\hbox to\hsize{\hskip.6in (a) $\sin(\theta)\,d\theta  d\varphi$  \quad \quad \quad  \quad \quad(b) $d\theta 
d\varphi$ \quad \quad \quad \quad \quad(c) $| \tan(\theta) |^{1/3}\, d\theta  d\varphi$  \hskip1in}
\caption{An illustration of i.i.d.\ samples from various spherical measures.  $(\pi, \varphi) \in [0,\pi) \times [0,2\pi]$.  The distribution (c) is the most incoherent with respect to the spherical harmonic basis.
}\label{sphere_sample}
\end{figure}

The spherical harmonics $Y_{\ell}^k$ form an orthonormal system for square-integrable functions on the sphere {$S^2 = \{\x \in \R^3: \|\x\|_2 = 1\}$}, 
and serve as a higher-dimensional analog of the univariate trigonometric basis.  
They are orthogonal with respect to the uniform spherical measure. {In spherical coordinates $(\varphi,\theta) \in [0,2 \pi) \times [0, \pi)$, $(x = \cos(\varphi)\sin(\theta),y = \sin(\varphi)\sin(\theta),z = \cos(\theta))$ for $(x,y,z) \in S^2$, the orthogonality reads}
\begin{equation}
\label{sphere:norm}
\int_{0}^{2\pi} \int_{0}^{\pi} Y_{\ell}^k (\varphi, \theta) \bar{Y}_{\ell'}^{k'}(\varphi, \theta) \sin(\theta) d\theta d\varphi = \delta_{\ell \ell', k k'}, \quad k, \ell \in \mathbb{Z}, \quad | k | \leq \ell.
\end{equation}
The spherical harmonics are bounded according to $\| Y_{\ell}^k \|_{L_\infty} \leq \ell^{1/2}$, and this bound is realized at the poles of the sphere, $\theta = 0, \pi$.  As shown in \cite{rw11, bdwz}, one can precondition the spherical harmonics to transform them into a system with smaller uniform bound, orthonormal with respect to a different measure.
For example, the preconditioned function system
$$
Z_{\ell}^k(\varphi, \theta) = \sin(\theta)^{1/2} Y_{\ell}^k (\varphi, \theta),
$$
normalized by the proper constant, is orthonormal on the sphere with respect to the measure $d\mu = d\theta d\varphi$ 
by virtue of \eqref{sphere:norm}.  The $Z_{\ell}^k$ are more uniformly bounded than the spherical harmonics $Y_{\ell}^k$; as noted in \cite{krasikov},
$$
\|  Z_{\ell}^k \|_{L_\infty} \leq C \ell^{1/4}
$$
for a universal constant $C$.  A sharper preconditioning estimate was shown in \cite{bdwz} for the system 
\begin{equation}
\label{sphere:sharper}
\tilde{Z}_{\ell}^k(\varphi, \theta) := (\sin^2(\theta) \cos(\theta))^{1/6} Y_{\ell}^k (\varphi, \theta).
\end{equation}
Normalized properly, this system is orthonormal on the sphere with respect to the measure $d \nu = |\tan(\theta)|^{1/3} d\theta d\varphi$, which is nonstandard and illustrated in Figure \ref{sphere_sample}. 
This system obeys the uniform bound
\begin{equation}
\label{weight:est}
\| \tilde{Z}_{\ell}^k \|_{\infty} \leq C \ell^{1/6},
\end{equation}
with $C$ a universal constant.  

We consider implications of Theorem \ref{thm:infRIPprob} for interpolation with spherical harmonic expansions.  We state a result in the setting where sampling points are drawn from the measure $ |\tan(\theta)|^{1/3} d\theta d\varphi$, but similar results (albeit with steeper weights) can be obtained for sampling from the measures $d\theta d\varphi$ and $\sin(\theta)\,d\theta  d\varphi$.

\begin{corollary}[Interpolation with spherical harmonics]
\label{thm:spherical}
Consider the  preconditioned spherical harmonics $\tilde{Z}_\ell^k,  | k | \leq \ell$, and associated orthogonalization measure $d\nu= |\tan(\theta)|^{1/3} d\theta d\varphi$.  
Fix weights $\omega_{\ell,k} = C\ell^{1/6}$ and index set $\Lambda_0 = \{(\ell,k) :  |k| \leq \ell \leq s^{3} \}$ of size $N = s^6$, and fix a number of samples  
$$m \geq c_0 s \log^4(s).$$ 
Consider a fixed function $f(\varphi, \theta) = \sum_{\ell,k} x_{\ell,k} \tilde{Z}_\ell^k(\varphi, \theta)  \in S_{\omega,1}$ {and} let $\eta = \triple f - f_{\Lambda_0} \triple_{\omega,1}$.  Draw $(\varphi_j, \theta_j), j=1, \dots, m,$ i.i.d.\ from $d\nu$, and consider sample values $y_j = f(\varphi_j, \theta_j)$, $j=1,\hdots,m.$  Then with probability exceeding $1-N^{-\log^3(s)}$, the function  $f^\sharp = \sum_{(\ell,k) \in \Lambda_0} x_{\ell,k}^\sharp \tilde{Z}_\ell^k$ formed from the solution $x^\sharp$ of the weighted $\ell_1$ minimization program
\[
\min_{u_{\ell,k}} \sum_{\ell,k \in \Lambda_0} \omega_{\ell,k}| u_{\ell,k} |  \mbox{ subject to }  \sum_{j=1}^m \Big( \sum_{\ell,k \in \Lambda_0} u_{\ell,k} \tilde{Z}_\ell^k (\varphi_j, \theta_j) - y_j \Big)^2 \leq \frac{m \eta^2}{s} 
\]
satisfies the error bounds
\begin{align}
\|f-f^\sharp\|_{\infty} \leq \triple f - f^\sharp \triple_{\omega,1} & \leq c_1 \sigma_s(f)_{\omega,1},\notag\\
\|f-f^\sharp\|_{L_2} & \leq c_2\sigma_s(f)_{\omega,1} / \sqrt{s}.\notag
\end{align}
\end{corollary}
It is informative to compare these results with previously available bounds for unweighted $\ell_1$ minimization.
Using the same sampling distribution $d \nu = |\tan(\theta)|^{1/3} d\theta d\varphi$ and number of basis elements $N = s^6$, existing bounds for unweighted $\ell_1$ minimization (see \cite{bdwz}) require a number of samples
$$
m \geq c  N^{1/6} s \log^4(s) {= c}  s^2 \log^4(s) 
$$
to achieve an error estimate of the form $\|f-f^\sharp\|_{L_2} \leq C \sigma_s(f)_{1}/\sqrt{s}$ (see \cite{rw11} for more details).  
That is, significantly more measurements $m$ are required to achieve the 
same reconstruction rate in $L_2$. (A rate for $L_\infty$ is not available for unweighted $\ell_1$-minimization).  However, stronger assumptions on $f$ are required 
in the sense that the result above requires the \emph{weighted} best $s$-term approximation error to be small while the bound from \cite{rw11} works
with the unweighted best $s$-term approximation error. Expressed differently, our result requires more smoothness which is in line with the general
philosophy of this paper.

 \subsection{Tensorized polynomial interpolation}
 
 The tensorized trigonometric polynomials on ${\cal D} = \T^d$ are given by
$$
\psi_{{\bf k}}({\bf t}) = \psi_{k_1}(t_1)\psi_{k_2}(t_2)\dots \psi_{k_d}(t_d), \quad {\bf k} \in \Z^d, 
$$
with $\psi_j(t) = e^{2\pi i j t}$.  These functions are orthonormal with respect to the tensorized uniform measure.
Because this system is uniformly bounded, Theorem \ref{thm:infRIPprob} applies with constant weights $\omega_j \equiv 1$.   Nevertheless, higher weights promote smoother reconstructions.

Other tensorized polynomial bases of interests are not uniformly bounded, but we can get reconstruction guarantees by considering weighted $\ell_1$ minimization with properly chosen weights.

\subsubsection{Chebyshev polynomials}

Consider the tensorized Chebyshev polynomials on ${\cal D} = [-1,1]^d$:  
\begin{equation}
\label{chebyshevT}
C_{{\bf k}}({\bf t}) = C_{k_1}(t_1)C_{k_2}(t_2) \dots C_{k_d}(t_d), \quad {\bf k} \in \N^d,
\end{equation}
where $C_{k}(t) = \sqrt{2} \cos\big( (k-1) \arccos(t) \big)$.
The Chebyshev polynomials form a basis for the real algebraic polynomials on ${\cal D}$, and are orthonormal with respect to the tensorized Chebyshev measure
\begin{equation}
\label{chebyshev}
d{\bf \mu} = \frac{d{\bf t}}{(2\pi)^d \Pi_{j=1}^d (1-t_j^2)^{1/2} }.
\end{equation}
Since $\| C_k \|_{\infty} = 2^{1/2}$ we have $\| C_{\bf k} \|_{\infty} = 2^\frac{ \| {\bf k} \|_0}{2}$.  Although the tensorized Chebyshev polynomials are uniformly bounded,
the bound may therefore scale exponentially in $d$. This motivates us to apply Theorem \ref{thm:infRIPprob} with weights 
\begin{equation}
\label{HyperCrossWeight}
\omega_{{\bf k}} 
{=} \prod_{\ell=1}^d (k_{\ell}+1)^{1/2},
\end{equation}
noting that $\| C_{\bf k} \|_{\infty} \leq \omega_{{\bf k}}$. {(More generally, one could also work with weights of the form 
$\omega_{\bf k} =  2^\frac{ \| {\bf k} \|_0}{2} v_{\bf k}$, where $v_{\bf k}$ tends to infinity as $\| \bf k \|_2 \to \infty$.)}
Such weights encourage both sparse and low order tensor products of Chebyshev polynomials.  The subset of indices
$$
H_s^d = \{ k \in {\N_0^d}: \omega_{{\bf k}}^2 \leq s \} = \left\{ k \in {\N_0^d}: \prod_{\ell=1}^d (k_{\ell}+1) \leq s \right\}
$$
forms a hyperbolic cross.  As argued in \cite[Proof of Theorem 4.9]{KSU13}, see also \cite{cdfr11},
the size of a hyperbolic cross can be bounded according to
$$
| H_s^d | \leq e^2 s^{2+\log_2(d)}.
$$
  
\begin{corollary}\label{thm:cheb}
Consider the tensorized Chebyshev polynomial basis $(C_{{\bf k}})$ for $[-1,1]^d$, and weights $\omega_{\bf k}$ as in \eqref{HyperCrossWeight}.  
Let {$\Lambda_0 = \{ {\bf k} \in \N_0^d : \omega^2_{\bf k} \leq s/2 \}$}, and let 
$N = | \Lambda_0 | \leq e^2 (s/2)^{2 + \log_2(d)}$.
Fix a number of samples 
\begin{align}\label{m:cheb:tensor}
{m \geq c_0 s \log^3(s) \log(N).} 
\end{align}
Consider a function $f  = \sum_{{\bf k} \in \Lambda} x_{\bf k} C_{\bf k},$ and sampling points ${\bf t}_\ell$, $\ell=1,\hdots,m$ drawn i.i.d. from the tensorized Chebyshev measure on $[-1,1]^d$.    Let  $\A \in \C^{m \times N}$ be the sampling matrix with entries $A_{\ell,j} = \psi_j(t_\ell)$.  From samples $y_\ell = f({\bf t}_\ell)$, $\ell=1,\hdots,m$, let $\x^\sharp$ be the solution of
\[
\min \|\z\|_{\omega,1} \mbox{ subject to }  \| \A \z - \y \|_2 \leq \sqrt{m/s} \triple f - f_{\Lambda_0} \triple_{\omega,1} 
\]
and set $f^\sharp({\bf t}) = \sum_{{\bf k} \in \Lambda_0} {x_{\bf k}} ^\sharp C_{\bf k}({\bf t})$. Then with probability exceeding $1 - N^{-\log^3(s)}$, 
\begin{align}
\| f - f^\sharp \|_{L_{\infty}} \leq \triple f - f^\sharp \triple_{\omega,1} & \leq c_1 \sigma_s(f)_{\omega,1}, \notag\\
\|f-f^\sharp\|_{L^2}  & \leq d_1\sigma_s(f)_{\omega,1} / \sqrt{s}.\notag
\end{align}
\noindent Above, $c_0, c_1,$ and $d_1$ are {universal constants.} 
\end{corollary}
\noindent 
Note that with the stated estimate of $N$, $m$ satisfies \eqref{m:cheb:tensor} once
\[
m \geq c_0 \log(d) s \log^4(s).
\]
This means that the required number of samples $m$ above grows only logarithmically with the ambient dimension $d$ 
as opposed to exponentially, as required for classical interpolation bounds using linear reconstruction methods.

\subsubsection{Legendre polynomials}

Consider now the tensorized Legendre polynomials on ${\cal D} = [-1,1]^d$:  

\begin{equation}
\label{legendreT}
L_{{\bf k}}({\bf t}) = L_{k_1}(t_1)L_{k_2}(t_2) \dots L_{k_d}(t_d), \quad {\bf k} \in \N^d,
\end{equation}
where $L_{k}$ is the univariate orthonormal Legendre polynomial of degree $k$.
The Legendre polynomials form a basis for the real algebraic polynomials on ${\cal D}$, and are orthonormal with respect to the tensorized uniform measure on ${\cal D}$.  Since $\| L_k \|_{\infty} \leq \sqrt{k}$ we have 
\begin{equation}
\label{inf:bound}
\| L_{\bf k} \|_{\infty} \leq \prod_{\ell=1}^d (k_{\ell}+1)^{1/2},
\end{equation}
and we may apply Theorem \ref{thm:infRIPprob} with hyperbolic cross weights $\omega_{{\bf k}} = \prod_{\ell=1}^d (k_{\ell}+1)^{1/2}$ as in Corollary \ref{thm:cheb}.  In doing so, we arrive at the following result.

\begin{corollary}\label{thm:leg}
Consider the tensorized Legendre polynomial basis and weights $\omega_{\bf k}$ as in \eqref{HyperCrossWeight} and with $\Lambda_0$, $N$, and $m$ as in Corollary \ref{thm:cheb}.  Consider a function $f  = \sum_{{\bf k} \in \Lambda} x_{\bf k} L_{\bf k},$ and suppose that ${\bf t}_\ell$, $\ell=1,\hdots,m$, are drawn i.i.d.\ from the tensorized uniform measure on $[-1,1]^{d}$.   Let  $\A \in \C^{m \times N}$ be the associated sampling matrix with entries $A_{\ell,j} = \psi_j(t_\ell)$. From samples $y_\ell = f({\bf t}_\ell)$, $\ell=1,\hdots,m$, let $\x^\sharp$ be the solution of
\[
\min \|\z\|_{\omega,1} \mbox{ subject to }  \| \A \z - \y \|_2 \leq \sqrt{m/s} \triple f - f_{\Lambda_0} \triple_{\omega,1} 
\]
and set $f^\sharp({\bf t}) = \sum_{{\bf k} \in \Lambda_0} {x_{\bf k}^\sharp} L_{\bf k}({\bf t})$. Then with probability exceeding $1 - N^{-\log^3(s)}$, 
\begin{align}
\| f - f^\sharp \|_{L_{\infty}} \leq \triple f - f^\sharp \triple_{\omega,1} & \leq c_1 \sigma_s(f)_{\omega,1}, \notag\\ 
\|f-f^\sharp\|_{L^2}  & \leq d_1\sigma_s(f)_{\omega,1} / \sqrt{s}.
\end{align}
\noindent Above, $c_0, c_1,$ and $d_1$ are universal constants. 
\end{corollary}
Although the univariate orthonormal Legendre polynomials are not uniformly bounded on $[-1,1]$, they can be transformed into a bounded orthonormal system by considering the weight 
$$
v(t) = (\pi/2)^{1/2} (1 - t^2)^{1/4}, \quad t \in [-1,1],
$$
and recalling Theorem 7.3.3 from \cite{szego} which states that, for all $j \geq 1$,
\begin{equation}
\label{leg:growth}
\sup_{t \in [-1,1]} v(t)| L_j(t) | \leq \sqrt{2+1/j} \leq \sqrt{3}.
\end{equation}
Then the preconditioned system $Q_j(t) = v(t) L_j(t)$ is orthonormal with respect to the Chebyshev measure, and is uniformly bounded on $[-1,1]$ with constant $K = \sqrt{3}$.  A statement similar to Corollary \ref{thm:cheb} can also be applied to tensorized preconditioned Legendre polynomials, if sampling points are chosen from the tensorized Chebyshev measure.   For further details, we refer the reader to \cite{rw10}.

{\subsection{Numerical illustrations}\label{numerics}}

Although the interplay between sparsity and smoothness in function approximation becomes more pronounced in high-dimensional problems, we still observe benefits of weighted $\ell_1$ minimization in univariate polynomial interpolation problems.  In this section we provide an illustration of this effect.

Polynomial interpolation usually refers to fitting the unique trigonometric or algebraic polynomial of degree $m-1$ through a given set of data of size $m$.  When $m$ is large, this problem becomes ill-conditioned, as illustrated for example by Runge's phenomenon, or the tendency of high-degree polynomial interpolants to oscillate at the edges of an interval (the analogous phenomenon for trigonometric polynomial interpolation is Gibb's phenomenon \cite{folland}).  One method for minimizing the effect of Runge's phenomenon is to  carefully choose interpolation nodes -- e.g., Chebyshev nodes for algebraic polynomial interpolation or equispaced nodes for trigonometric interpolation.  Other methods known to reduce the effects of Runge's phenomenon are the method of least squares, where one foregoes exact interpolation for a least squares projection of the data onto a polynomial of lower degree and which has been shown to be stable with respect to random sampling \cite{cdd11}, or by doing weighted $\ell_2$ regularization  \cite{kunis2007stability}, e.g.\ use for interpolation the function $f^{\sharp}(t) = \sum_{j \in \Lambda} x_j^{\sharp} \psi_j(t)$, where the coefficient vector $\x^{\sharp}$ solves the minimization problem
\begin{equation}
\label{eq:wl2}
\min \sum_{j \in \Lambda} \alpha_j^2 z_j^2 \mbox{ subject to } \A \z = \y
\end{equation}
with $\A$ being the sampling matrix as in \eqref{sample_mat}.

In this section we provide evidence that weighted $\ell_1$ minimization can be significantly less sensitive to perturbations in the choice of the sampling points than these other methods, specifically unweighted $\ell_1$ minimization, least squares projections, weighted $\ell_2$ regularization, and exact polynomial interpolation.

For our numerical experiments, we follow the examples in \cite{cdd11} and consider on $[-1,1]$ the smooth function 
$$
f(t) = \frac{1}{1+ 25t^2},
$$
which was originally considered by Runge \cite{runge} to illustrate the instability of polynomial interpolation at equispaced points. We repeat the following experiment 100 times: draw $m = 30$ sampling points $x_1, x_2, \dots, x_m,$ i.i.d. from a measure $\mu$ on ${\cal D} = [-1,1]$ and compute the noise-free observations $y_k = f(x_k)$.  We will use the uniform measure for real trigonometric polynomial interpolation and the Chebyshev measure  for Legendre polynomial interpolation. We then compare  the least squares approximation, unweighted $\ell_1$ approximation, weighted $\ell_2$ approximation with weights $\alpha_j = j$ in \eqref{eq:wl2}, and weighted $\ell_1$ approximations with weights $\omega_j = j$ and weights $\omega_j = j^{1/2}.$  We also compare to exact inversion, where we fit the sampling points with a polynomial whose maximal degree -1 (degrees of freedom) matches the number of sampling points. 
In Figures~\ref{fig:1}-\ref{fig:2} we display the interpolations resulting from all 100 experiments overlaid so as to illustrate the variance of each interpolation method with respect to {the} choice of sampling points. In all experiments, we fix in the $\ell_1$ and $\ell_2$ minimization a maximal polynomial degree $N=100$.  For the least squares solution to be stable \cite{cdd11}, 
{we project} onto the span of the first $d = 15$ basis elements.

  In this example, we observe that the weighted $\ell_1$ solutions are more consistently accurate than the other methods, including exact polynomial interpolation.  There is mild sparsity present in this example; Runge's function is an even function and so all odd coefficients in its Legendre / trigonometric expansion are zero.  Indeed, the weighted $\ell_1$ solutions tend to pick up this sparsity and have zero odd coefficients, unlike the other interpolation methods which do not pick up on this sparsity. 
  
We emphasize once more that this is just an illustration.  For univariate polynomial interpolation in the setting where one can design carefully-chosen sampling points, e.g., Chebyshev nodes, weighted $\ell_1$ minimization should not outperform exact polynomial interpolation.

\begin{figure}[H]
\begin{center}

\subfigure[Original function]{
 \includegraphics[height=2.5cm,width=4cm]{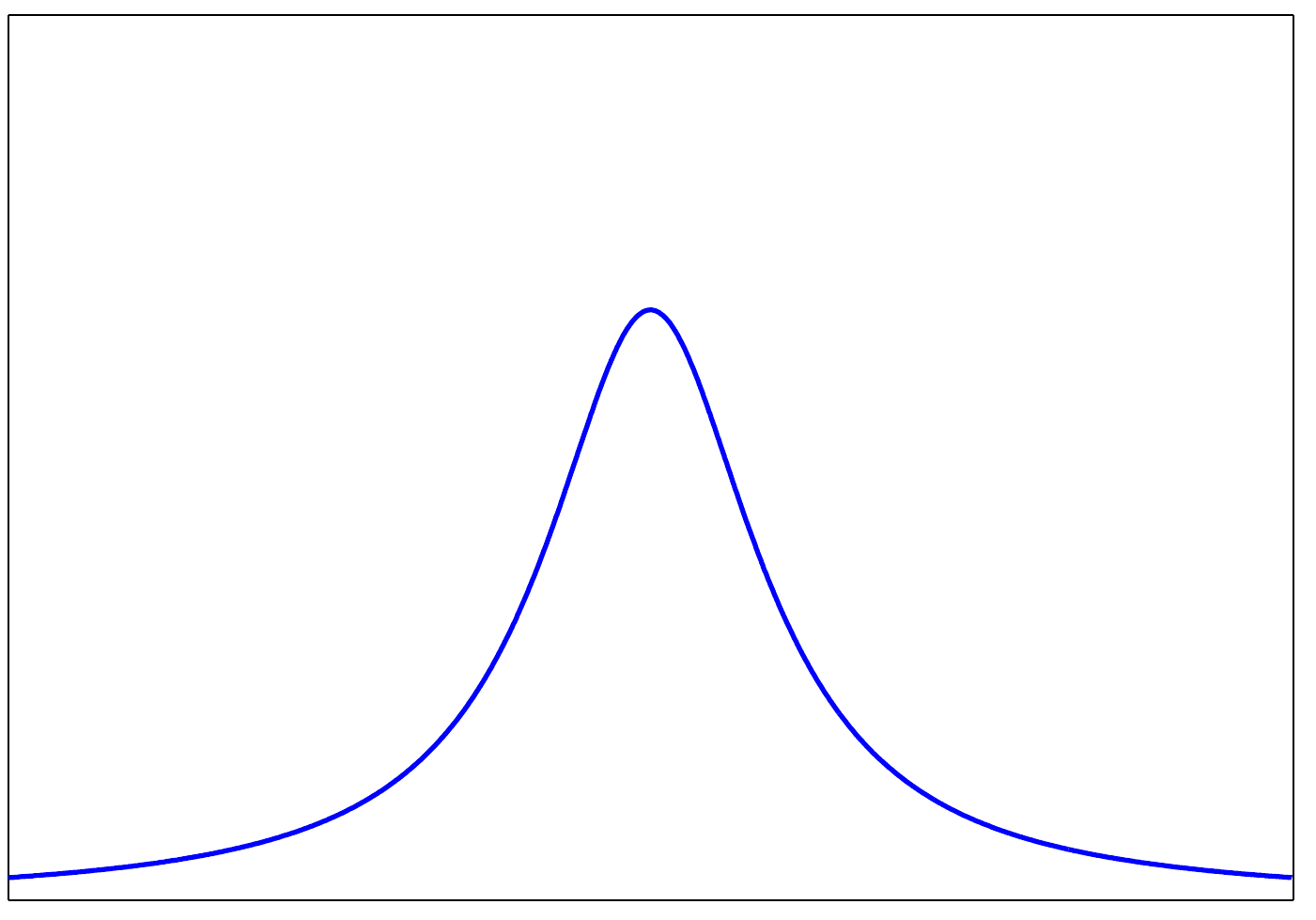}
}

\subfigure[Least squares]{
 \includegraphics[height=2.5cm,width=4cm]{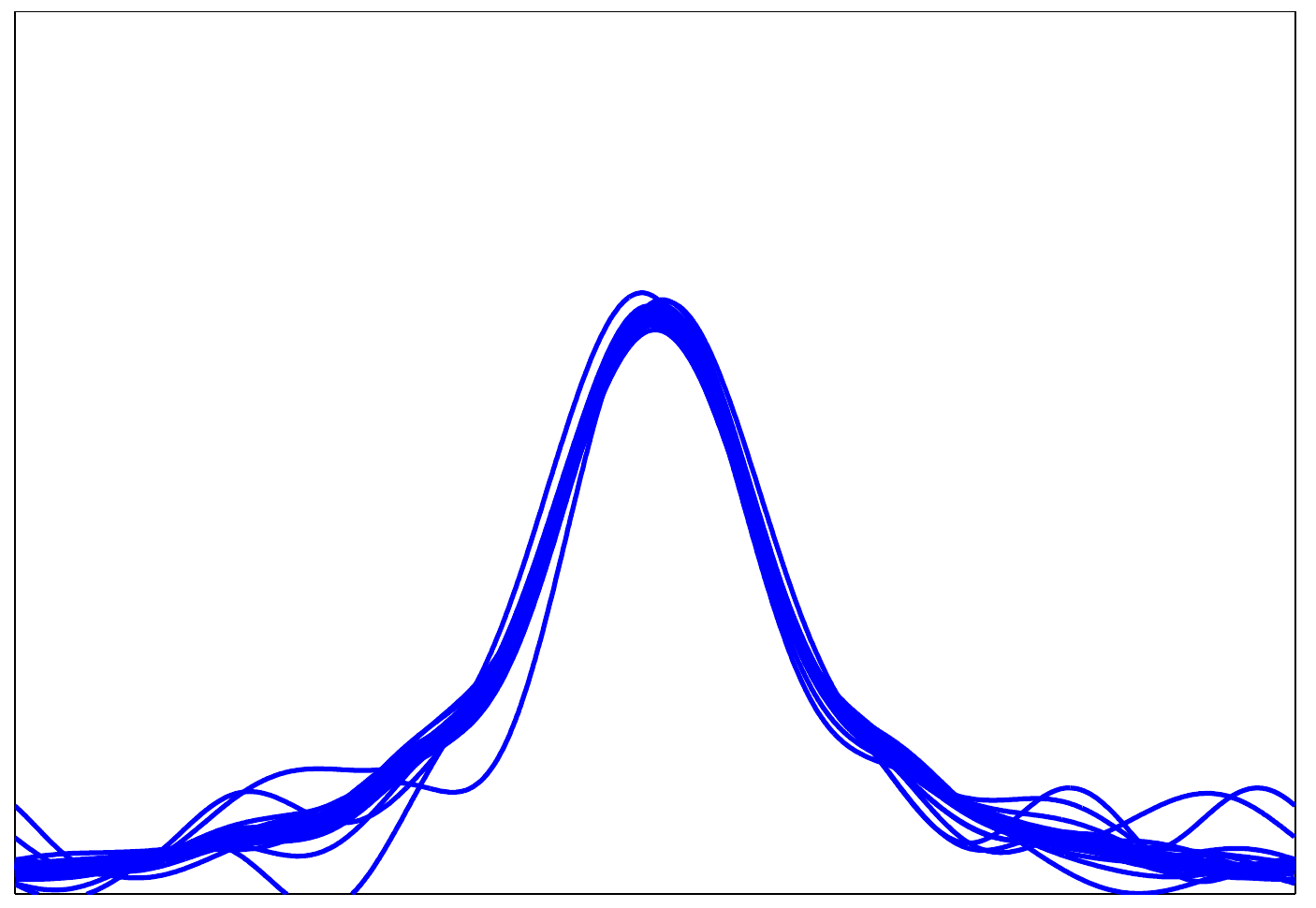} 
 }
 \subfigure[Weighted $\ell_2$, $\omega_j = j^{1/2}$]{
 \includegraphics[height=2.5cm,width=4cm]{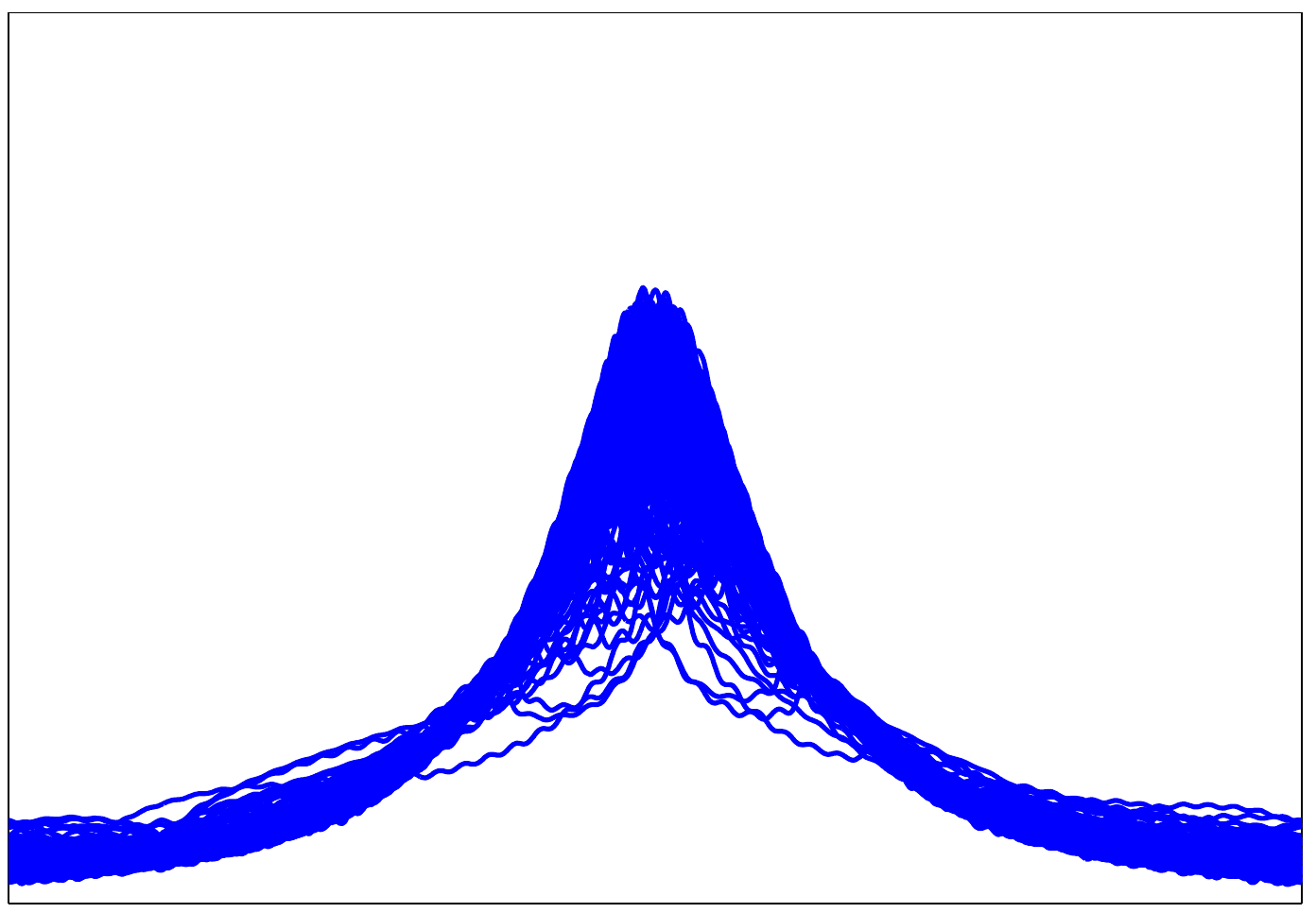} 
 }
 \subfigure[Exact inversion]{
  \includegraphics[height=2.5cm,width=4cm]{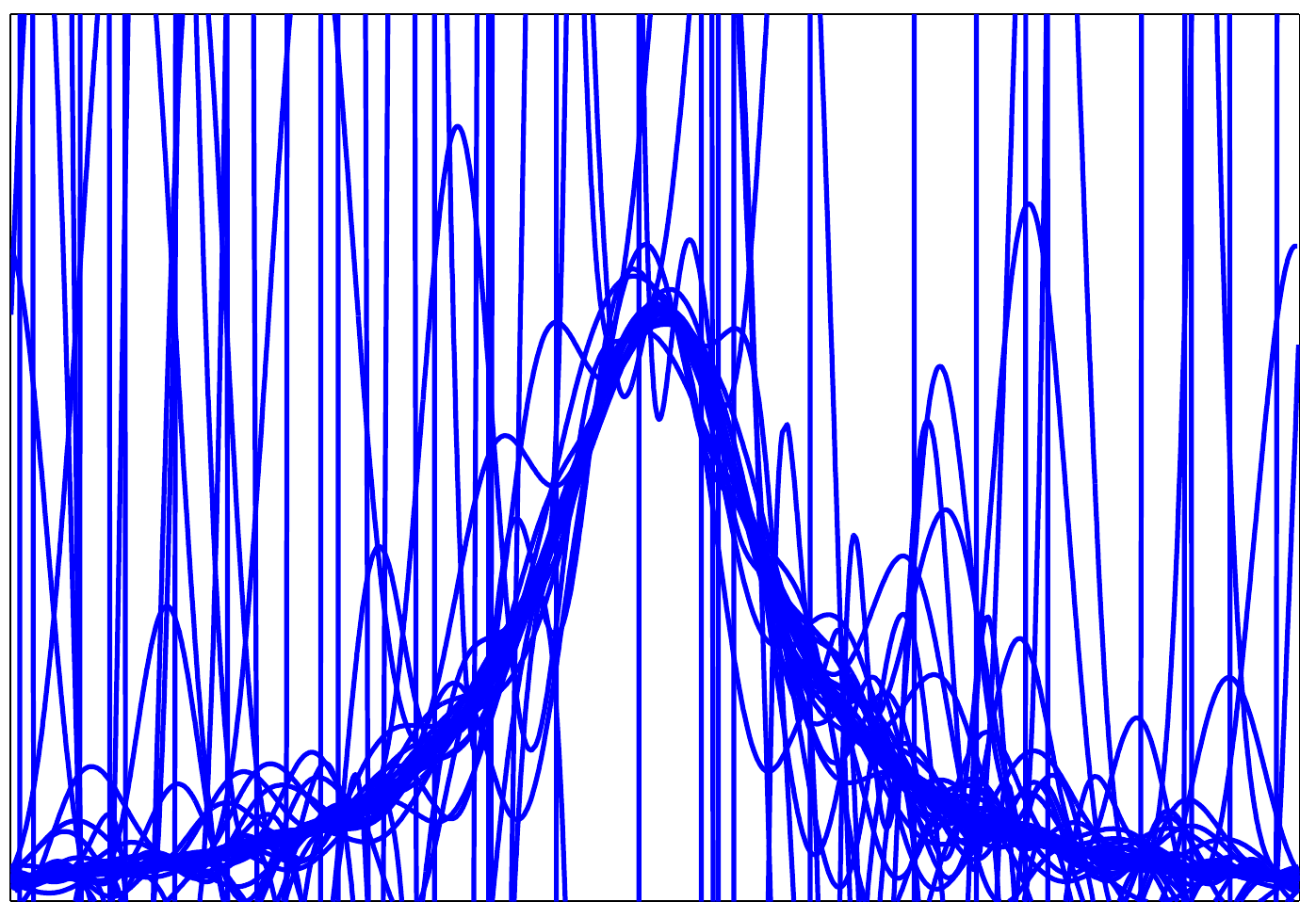} 
}

\subfigure[Unweighted $\ell_1$]{
    \includegraphics[height=2.5cm,width=4cm]{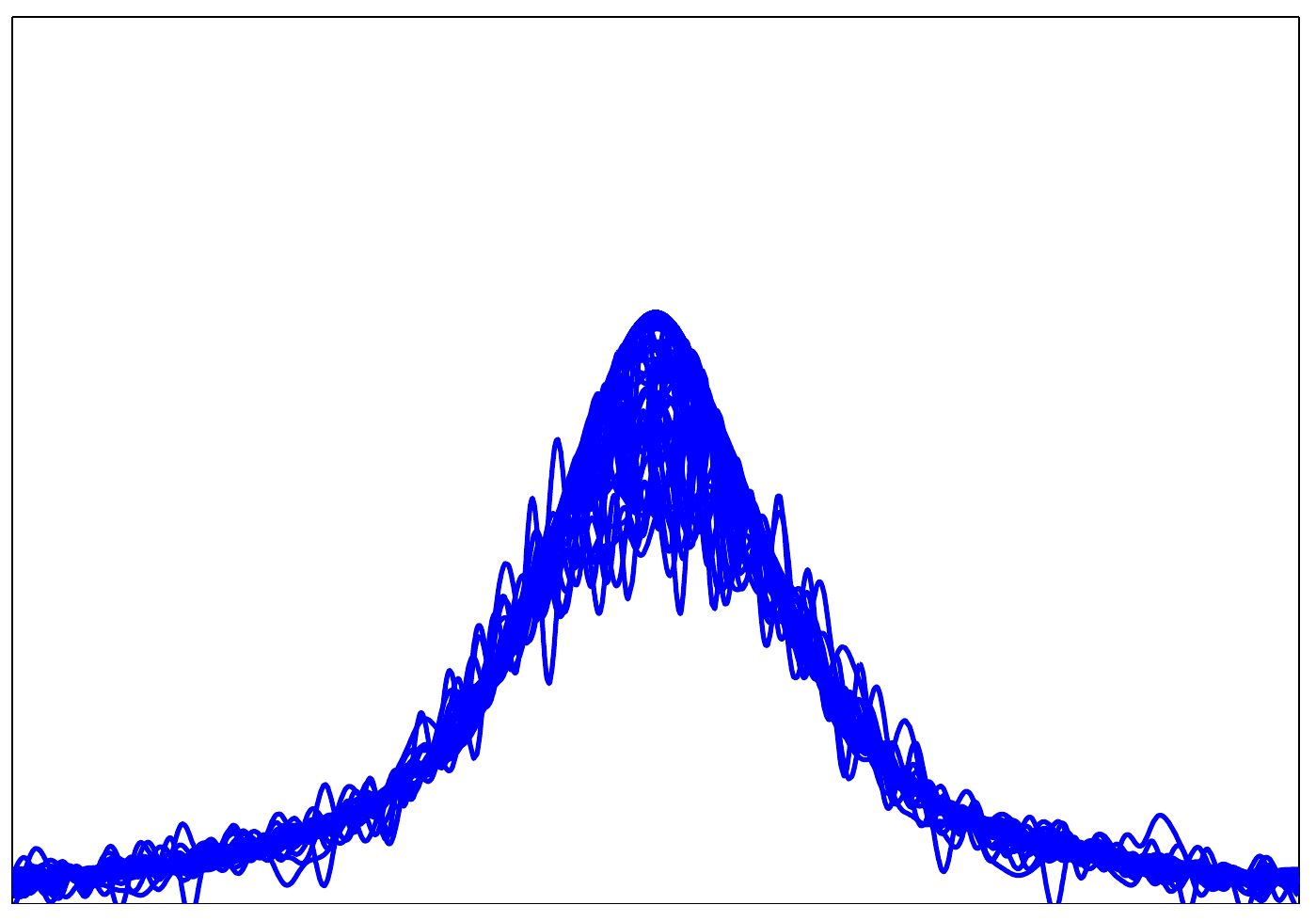} 
    }
    \subfigure[Weighted $\ell_1$, $\omega_j = j^{1/2}$]{
 \includegraphics[height=2.5cm,width=4cm]{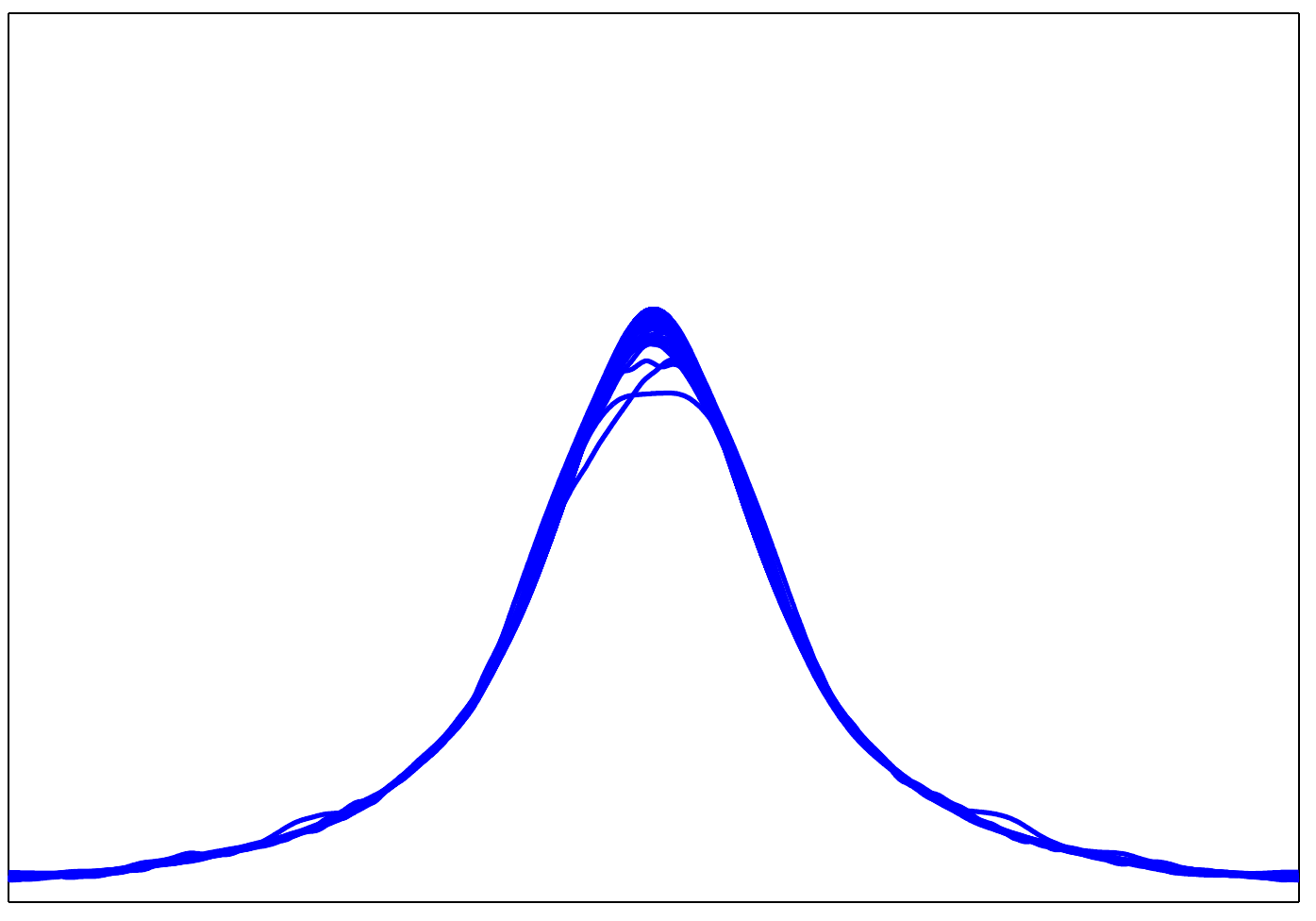}
 }
 \subfigure[Weighted $\ell_1$, $\omega_j = j$]{
  \includegraphics[height=2.5cm,width=4cm]{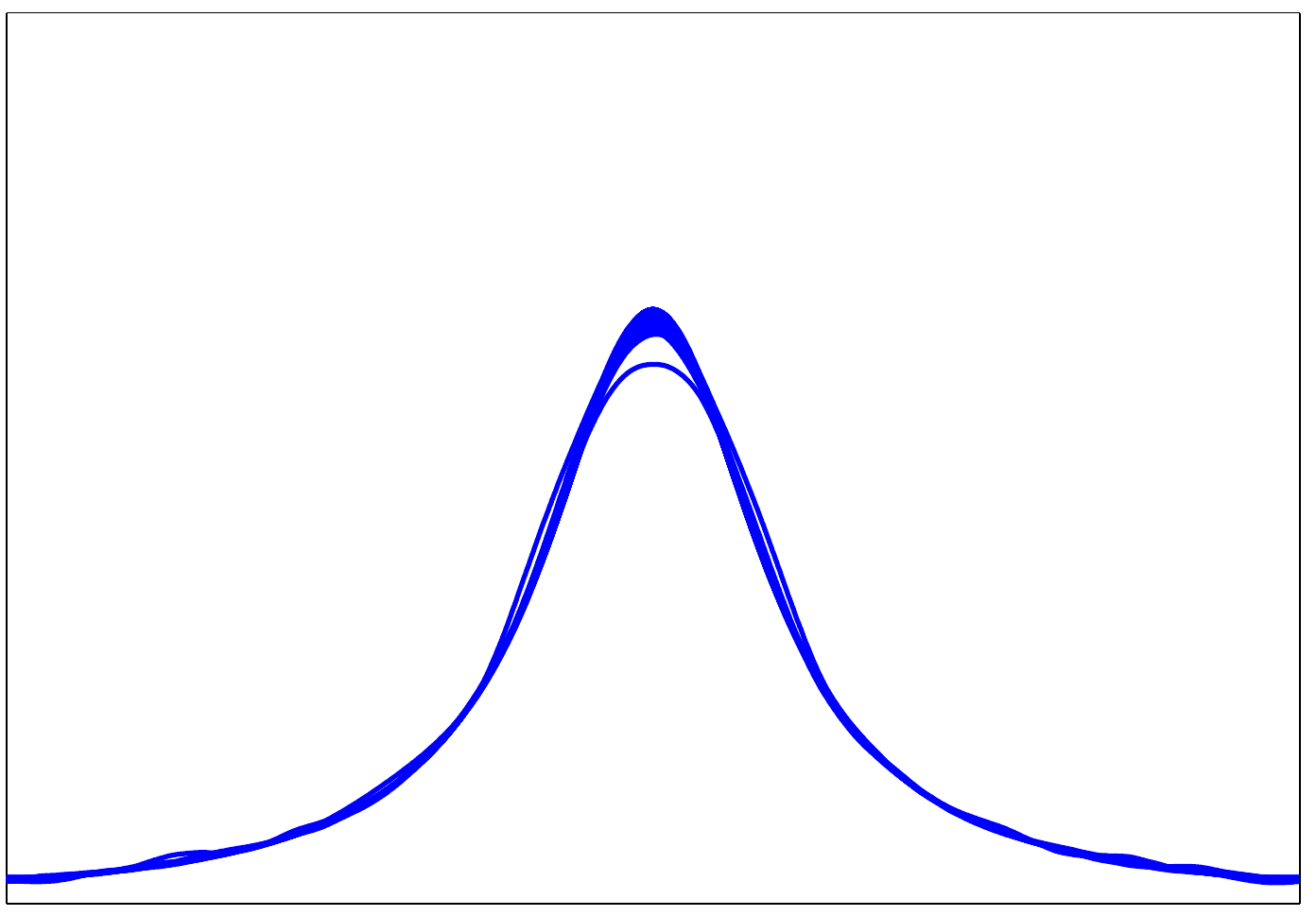}
}
\end{center}
\caption{Overlaid interpolations of the function $f(t) = \frac{1}{1+ 25t^2}$, $t \in [-1,1]$, by real trigonometric polynomials using various reconstruction methods as described in Section \ref{numerics}.  Different interpolations correspond to different  random draws of $m=30$ sampling points from the uniform measure on $[-1,1]$.}

\label{fig:1}
\end{figure} 

\begin{figure}[h!]
\begin{center}


\subfigure[Least squares]{
  \includegraphics[height=2.5cm,width=4cm]{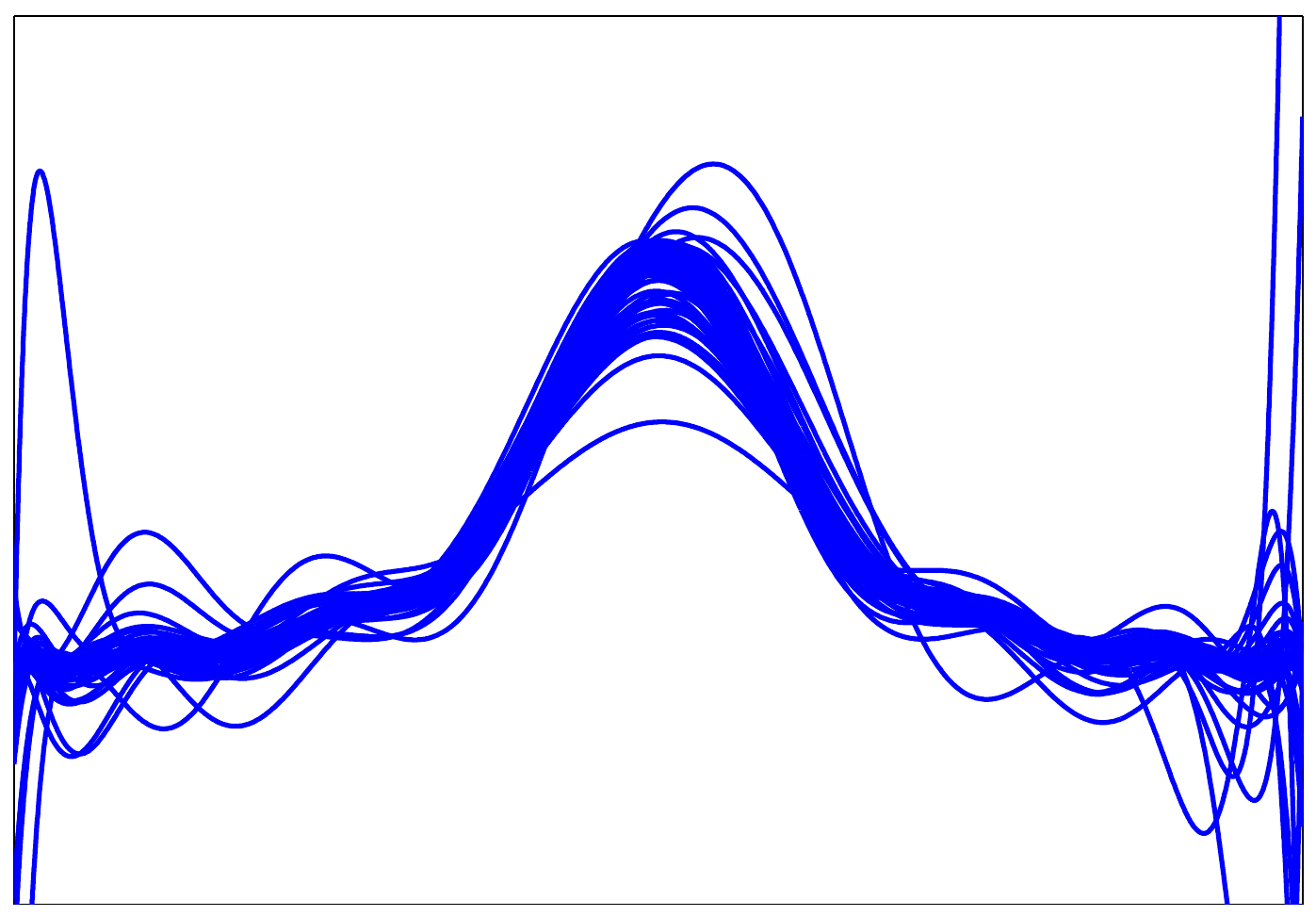}
  }
 \subfigure[Weighted $\ell_2$, $\omega_j = j^{1/2}$]{
 \includegraphics[height=2.5cm,width=4cm]{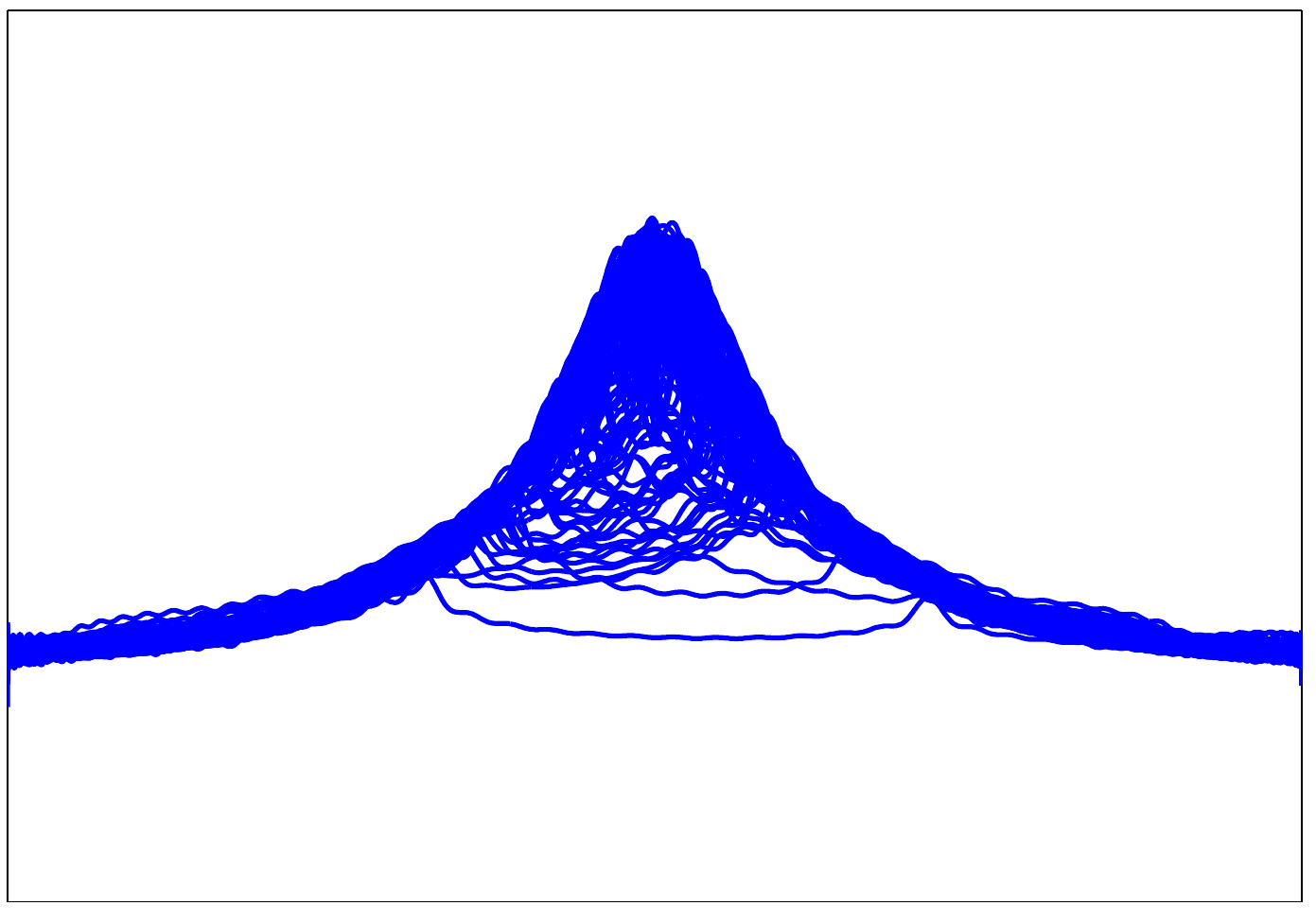} 
 }
 \subfigure[Exact inversion]{
  \includegraphics[height=2.5cm,width=4cm]{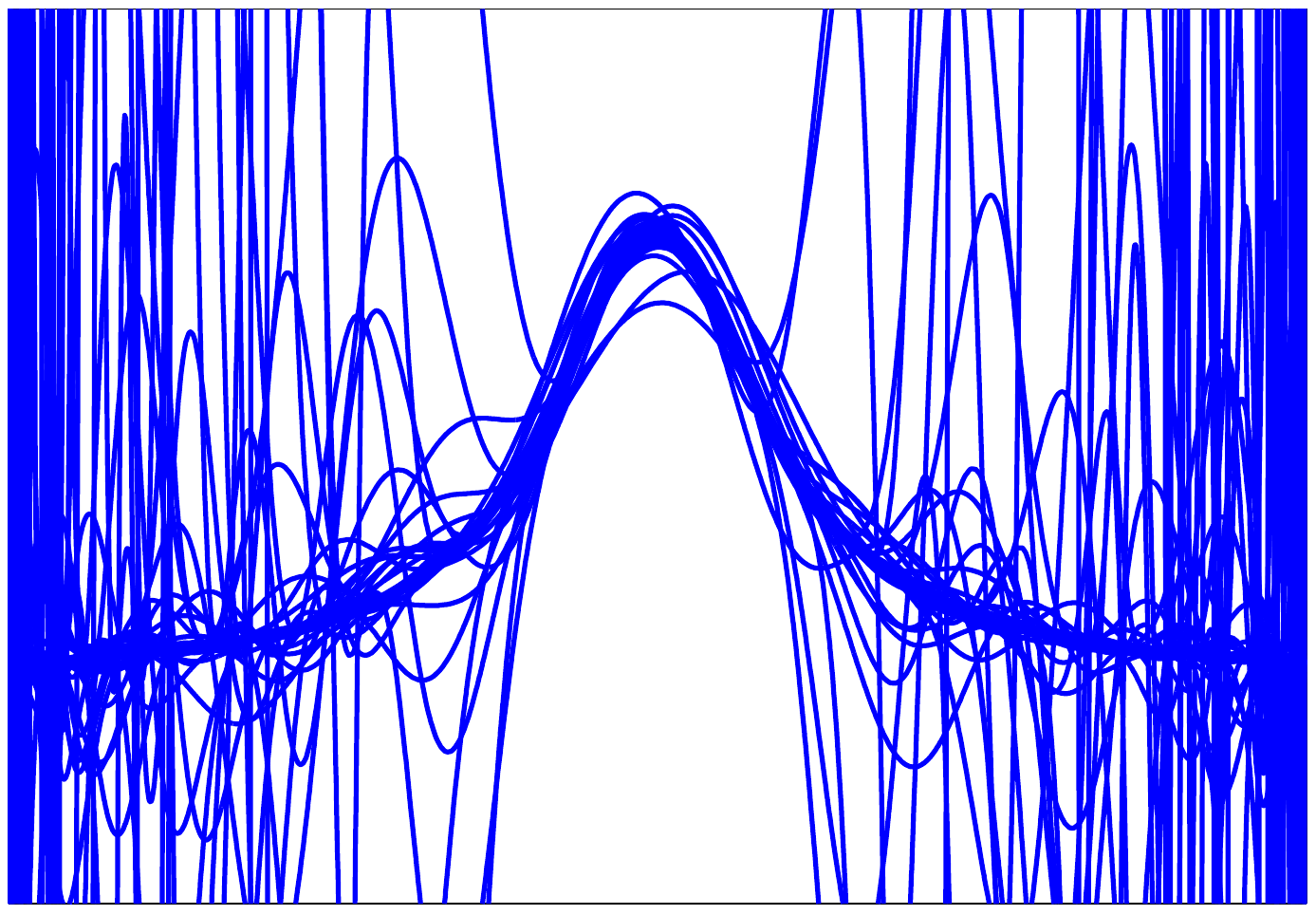} 
}

\subfigure[Unweighted $\ell_1$]{
    \includegraphics[height=2.5cm,width=4cm]{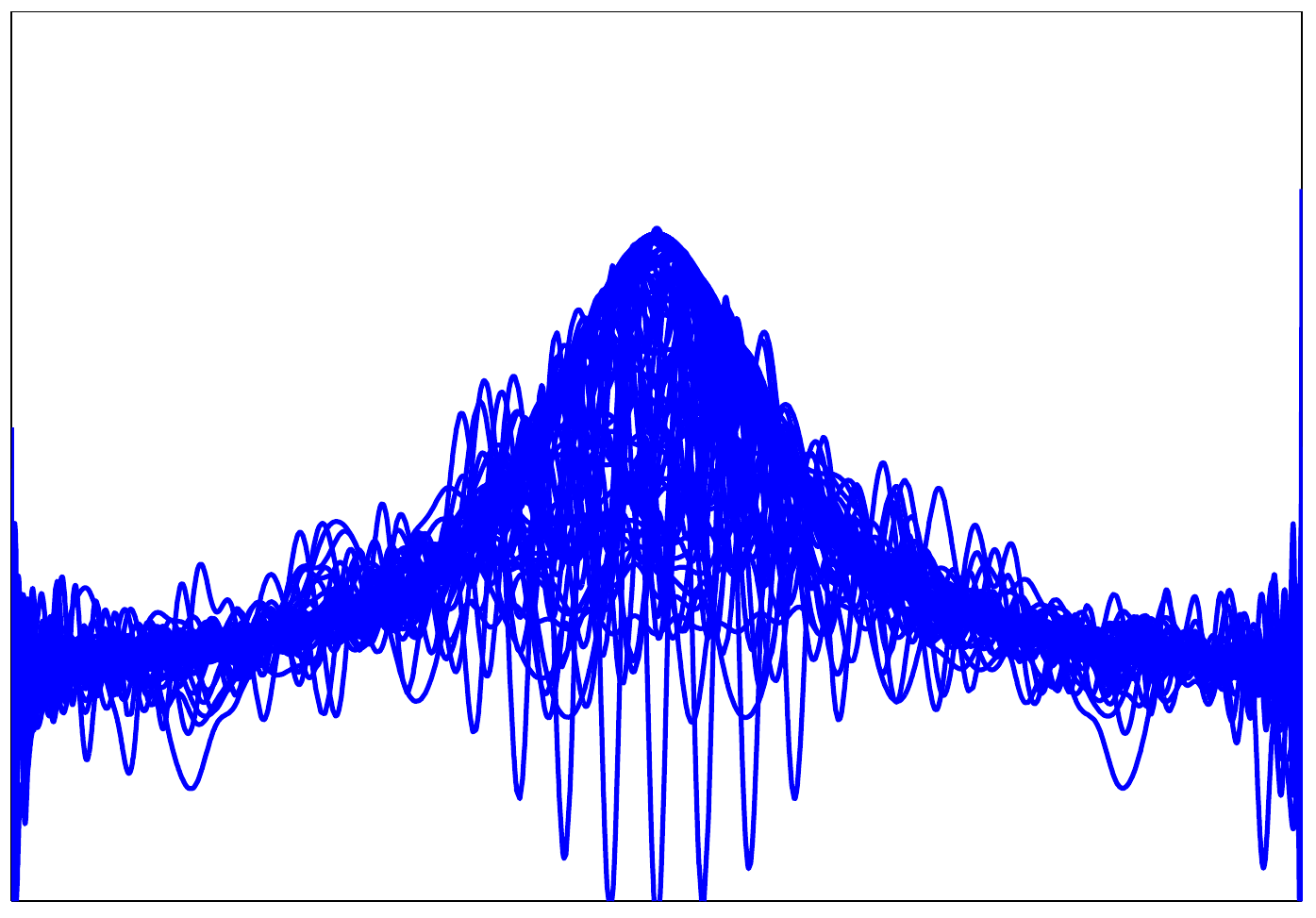} 
    }
    \subfigure[Weighted $\ell_1$, $\omega_j = j^{1/2}$]{
 \includegraphics[height=2.5cm,width=4cm]{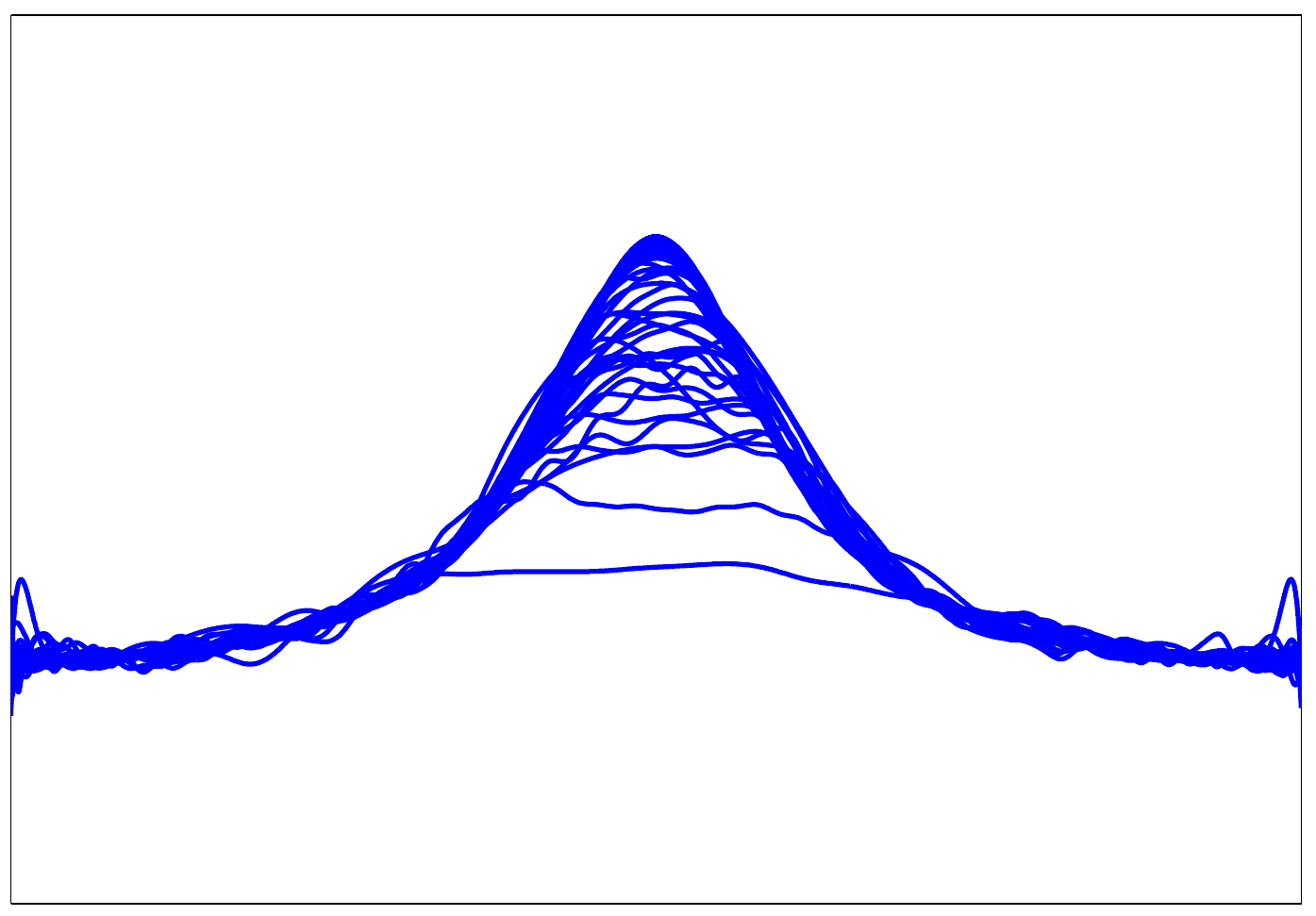}
 }
 \subfigure[Weighted $\ell_1$, $\omega_j = j$]{
  \includegraphics[height=2.5cm,width=4cm]{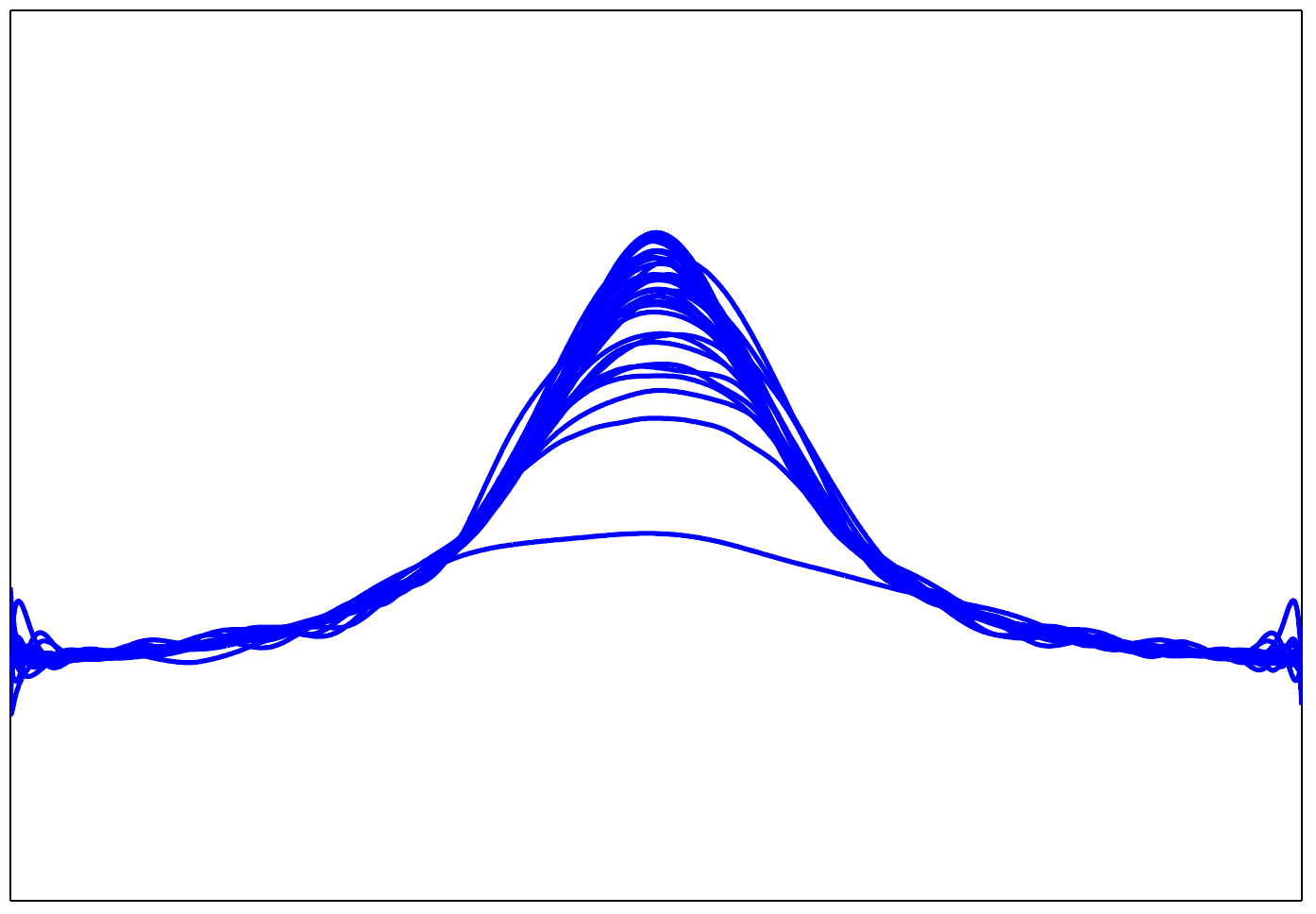}
}
\end{center}
\caption{Overlaid interpolations of the function $f(t) = \frac{1}{1+ 25t^2}$ by Legendre polynomials using various reconstruction methods as described in Section \ref{numerics}.  Different interpolations correspond to different  random draws of $m=30$ sampling points according the Chebyshev measure on $[-1,1]$.}
\label{fig:2}
\end{figure}

{\section{Weighted sparsity and quasi-best $s$-term approximations}\label{stechkin}}

In this section we revisit some important technical results pertaining to weighted $\ell_p$ spaces that were touched upon in the introduction.  First, unlike unweighted $s$-term approximations for finite vectors, the weighted $s$-term approximations $\sigma_s(\x)_{\omega, p} = \inf_{\z : \|\z\|_{\omega, 0} \leq s} \|\x-\z\|_{\omega, p}$ are not straightforward 
to compute in general.  Nevertheless, we can approximate $\sigma_s(\x)_{\omega,p}$ using a quantity that can easily be computed from $\x$ by sorting and thresholding, which we will call the quasi-best $s$-term approximation.

Let $\vv$ denote the non-increasing rearrangement of the sequence  
$(|x_j|^p \omega_j^{-p})$, that is, $v_j = | x_{\pi(j)} |^p \omega_{\pi(j)}^{-p}$ for some permutation $\pi$ such that $v_1 \geq v_2 \geq \dots \geq 0$. Let $k$ be the maximal number such that $\sum_{j=1}^k \omega_{\pi(j)}^2 \leq s$ and set $S= \{\pi(1),\pi(2),\hdots,\pi(k)\}$ so that
$\omega(S) \leq s$.  Then we call $\x_S$ a weighted quasi-best $s$-term approximation to $\x$ and define the corresponding error of weighted quasi-best
$s$-term approximation as
\[
\widetilde{\sigma}_s(\x)_{\omega,p} = \|\x-\x_S\|_{\omega,p} = \|\x_{S^c}\|_{\omega,p}.
\] 
By definition, $\sigma_s(\x)_{\omega, p} \leq \widetilde{\sigma}_s(\x)_{\omega,p}$. We also have a converse inequality relating the two $s$-term approximations in the case of bounded weights.

\begin{lemma}
\label{best:quasi}
Suppose that {$ s \geq \| \w \|^2_{\infty}$}. Then
$$
\widetilde{\sigma}_{3s}(\x)_{\omega,p} \leq \sigma_s(\x)_{\omega,p}
$$
\end{lemma}

\begin{proof}
Let $\vct{x}_S$ be the weighted best $s$-term approximation to $\vct{x}$, and let $\vct{x}_{\tilde{S}}$ be the weighted quasi-best $3s$-term approximation to $\vct{x}$.  Because the supports of $\vct{x}_S$ and $\vct{x} - \vct{x}_S$, and also $\vct{x}_{\tilde{S}}$ and $\vct{x} - \vct{x}_{\tilde{S}}$, do not overlap, it suffices to show that
$$
\| \vct{x}_{S} \|_{\omega,p} \leq \| \vct{x}_{\tilde{S}} \|_{\omega,p}.
$$
Assume without loss of generality that the terms in $\vct{x}$ are ordered so that $|x_j|^p \omega_j^{-p} \geq |x_{j+1}|^p \omega_{j+1}^{-p}$ for all $j$; let $J$ be the largest integer such that $\sum_{j=1}^J \omega^2_j \leq 3s$, and $\tilde{S} = \{1,2, \dots, J \}$.  Because $| \omega_j |^2 \leq s,$ we know also that $\sum_{j=1}^J \omega^2_j \geq 2s$.  Let $n_j = \lfloor \omega_j^2 + 1 \rfloor$ be the largest integer less than or equal to $\omega_j^2 + 1$, and let $r_j = n_j - \omega_j^2$.   Then $ \sum_{j \in S} \omega_j^2 \leq s$ implies that
\begin{equation}
\sum_{j \in S} n_j \leq \sum_{j \in S} \omega_j^2 + |S| \leq s + s \leq \sum_{j=1}^J \omega^2_j \leq \sum_{j=1}^J n_j.
\end{equation}
Now, let $\vct{z}$ be the vector
$$
\big( \underbrace{ |x_1|^p \omega_1^{-p}, \dots, |x_1|^p \omega_1^{-p}, (1-r_1) |x_1|^p \omega_1^{-p} }_\text{$n_1$ coefficients }, \hspace{1mm} \underbrace{|x_2|^p \omega_2^{-p}, \dots, |x_2|^p \omega_2^{-p}, (1-r_2) |x_2|^p \omega_2^{-p} }_\text{$n_2$ coefficients}, \dots \big).$$
We constructed $\vct{z}$ so that the first $n_1$ terms in $\vct{z}$ sum to $|x_1|^p \omega_1^{2-p}$, the next $n_2$ terms sum to  $|x_2|^p \omega_2^{2-p},$ and so on. 
 Then 
\begin{align}
\| \vct{x}_{S} \|_{\omega,p}^p &:= \max \left\{ \sum_{j \in S} \omega^2_j  | x_j |^p \omega_j^{-p} : S  \text{ is such that } \sum_{j \in S} \omega_j^2 \leq s \right\} \nonumber \\
&\leq \max \left\{ \sum_{j \in S} \omega^2_j  | x_j |^p \omega_j^{-p} : S \text{ is such that } \sum_{j \in S} n_j \leq \sum_{j=1}^J n_j \right\}  \nonumber \\
&\leq \max \left\{ \sum_{k \in \Lambda}  z_k : \Lambda \text{ is such that }  \# \Lambda \leq \sum_{j=1}^J n_j \right\} 
\leq \sum_{j=1}^{J} \omega_j^{2-p} |x_j|^p = \| \vct{x}_{\tilde{S}} \|_{\omega,p}^p \nonumber
\end{align}%
where in the last line the maximum is taken over all $\Lambda$ of the form $\Lambda = \cup_{k} \Lambda_k$, where each $\Lambda_k$ is a block of the first $n_1$ indices, or the second $n_2$ indices etc. and such that $|\Lambda| \leq \sum_{j=1}^J n_j$.  This completes the proof.
\end{proof}
In the remainder of this section, we prove the Stechkin-type estimate \eqref{stech1} which bounds the quasi-best $s$-term approximation of a vector (and hence also the best $s$-term approximation) using an appropriate weighted vector norm.

\begin{theorem}\label{thm:weighted:Stechkin} For $p < q \leq 2$, let $\x \in \ell_{\omega,p}$. Then, for {$s > \|\w\|_\infty^2$},
\begin{equation}\label{weighted:Stechkin}
\sigma_{s}(\x)_{\omega,q} \leq \widetilde{\sigma}_{s}(\x)_{\omega,q} \leq \big( s - \|\w\|_\infty^2 \big)^{1/q-1/p} \|\x\|_{\omega,p}.
\end{equation}
\end{theorem}
\begin{proof}
Let $S$ be the support of the weighted quasi-best $s$-term approximation, so that 
$\widetilde{\sigma}_s(\x)_{\omega,p} = \|\x-\x_S\|_{\omega,p}$. Since the number $k$ in the construction of $S$ is maximal, we have with $\pi$ denoting the corresponding permutation,
\[
s-\|\w\|_\infty^2 \leq s - \omega_{\pi(k)}^2 < \omega(S) \leq s. 
\]
Then
\[
\widetilde{\sigma}_s(\x)_{\omega,p}^p \leq \sum_{j \notin S} |x_j|^p \omega^{2-p}_j \leq
\max_{j \notin S} \{ |x_j|^{p-q} \omega_j^{q-p} \} \sum_{j \notin S} |x_j|^{q} \omega_j^{2-q}
\leq \left( \max_{j \notin S} |x_j| \omega_j^{-1} \right)^{p-q} \|\x\|_{\omega,q}^{q}.
\]
Now let $\alpha_k := (\sum_{j \in S} \omega_j^2)^{-1} \omega_k^2 \leq (s-\| \w\|_{\infty}^2)^{-1} \omega_k^2$.
Then $\sum_{j \in S} \alpha_k =1$. Moreover, by definition of $S$ we have $|x_j| \omega_j^{-1} \leq |x_k| \omega_k^{-1}$ for all
$k \in S$ and $j \notin S$. This implies
\[
\left(\max_{j \notin S} |x_j| \omega_j^{-1}\right)^q \leq \sum_{k \in S} \alpha_k |x_k|^{q} \omega_k^{-q}
\leq (s-\| \w\|_{\infty}^2)^{-1} \sum_{k \in S} \omega_k^{2-q} |x_k|^q
\leq (s-\| \w\|_{\infty}^2)^{-1} \|\x\|_{\omega,q}^q.
\]
Combining the above estimates yields
\[
\widetilde{\sigma}_s(\x)_{\omega,p}^p \leq  \left((s-\| \w\|_{\infty}^2)^{-1} \|\x\|_{\omega,q}^q\right)^{(p-q)/q} \|\x\|_{\omega,q}^{q}
\]
which is equivalent to the claim.
\end{proof}
Theorem \ref{thm:weighted:Stechkin} will be used in deriving weighted null space properties and weighted restricted isometry properties in the following sections.

\section{Weighted null space and restricted isometry property}

As for unweighted $\ell_1$ minimization, one can derive reconstruction guarantees for weighted $\ell_1$ minimization via appropriate weighted versions of the null space property and restricted isometry property \cite{cohen2009, cata06}. Below we work out these approaches.

\subsection{Weighted null space property}

We start directly with a {robust} version of the null space property in the weighted case.

\begin{definition}[Weighted robust null space property]
Given a weight $\w$, a matrix $\A \in \C^{m \times N}$ is said to satisfy the weighted 
robust null space property of order $s$ with constants $\rho \in (0,1)$ and $\tau > 0$ if
\begin{align}\label{robustNSP}
\|\vv_{S}\|_{2} \leq \frac{\rho}{\sqrt{s}} \|\vv_{S^c}\|_{\omega,1} + \tau \|\A \vv\|_2 \quad \mbox{ for all } \vv \in \C^{N} \mbox{ and all } 
S \subset [N] \mbox{ with } \omega(S) \leq s.
\end{align}
\end{definition}
The inequalities stated in the next theorem are crucial for deriving error bounds for recovery via weighted $\ell_1$ minimization.
\begin{theorem}\label{thmNSP} Suppose that $\A \in \C^{m \times N}$ is such that \eqref{robustNSP} holds for $\rho \in (0,1)$ and $\tau > 0$. 
Then, for all $\x,\z \in \C^N$, we have
\begin{align}\label{l1errorbound}
\|\z-\x\|_{\omega,1} \leq \frac{1+\rho}{1-\rho}\left(\|\z\|_{\omega,1} - \|\x\|_{\omega,1} + 2 \sigma_{s}(\x)_{\omega,1}\right) + \frac{2\tau\sqrt{s}}{1-\rho} \|\A(\z-\x)\|_2
\end{align}
and, additionally assuming $s \geq 2 \|\w \|_\infty^2$, 
\begin{equation}\label{l2errorbound}
\|\x-\z\|_2 \leq \frac{C_1}{\sqrt{s}} \left(\|\z\|_{\omega,1} - \|\x\|_{\omega,1} + 2 \sigma_{s}(\x)_{\omega,1}\right) + C_2 \|\A(\x-\z)\|_2.
\end{equation}
\end{theorem}
\begin{proof} We start with the proof of \eqref{l1errorbound}.
Let $S$ with $\omega(S) \leq s$ {be} such that $\sigma_{s}(\x)_{\omega,1} = \|\x-\x_S\|_{\omega,1} = \|\x_{S^c}\|_{\omega,1}$.
The triangle inequality gives 
\begin{align}
& \|\x\|_{\omega,1} + \|(\x-\z)_{S^c}\|_{\omega,1}  \leq \|\x_{S^c}\|_{\omega,1} + \|\x_{S}\|_{\omega,1} + \|\x_{S^c}\|_{\omega,1} + 
\|\z_{S^c}\|_{\omega,1}\notag\\
&\leq 2 \|\x_{S^c}\|_{\omega,1} + \|(\x-\z)_S\|_{\omega,1} + \|\z_S\|_{\omega,1} + \|\z_{S^c}\|_{\omega,1}
= 2 \sigma_s(\x)_{\omega,1} + \|(\x-\z)_S\|_{\omega,1} + \|\z\|_{\omega,1}.\notag
\end{align}
Rearranging and setting $\vv := \z-\x$ leads to
\begin{equation}\label{aux:NSP}
\|\vv_{S^c}\|_{\omega,1} \leq \|\z\|_{\omega,1} - \|\x\|_{\omega,1} + \|\vv_S\|_{\omega,1} + 2\sigma_s(\x)_{\omega,1}.
\end{equation}
The Cauchy-Schwarz inequality implies 
\[
\| \vv_S  \|_{\omega,1} = \sum_{j \in S} |v_j| \omega_j \leq \sqrt{\sum_{j \in S} |v_j|^2}\sqrt{\sum_{j \in S} \omega_j^2} = \sqrt{\omega(S)} \| \vv_S \|_2 \leq \sqrt{s} \| \vv_S\|_2,
\]
and therefore by \eqref{robustNSP} 
\begin{equation}\label{l1NSP}
\| \vv_S \|_{\omega,1} \leq \sqrt{s} \|\vv_S\|_2  \leq \rho \| \vv_{S^c} \|_{\omega,1} + \tau \sqrt{s} \| \A \vv \|_2. 
\end{equation}
We combine with \eqref{aux:NSP} 
to arrive at
\[
\|\vv_{S^c}\|_{\omega,1} \leq \frac{1}{1-\rho} \left(\|\z\|_{\omega,1} - \|\x\|_{\omega,1}  + \tau \sqrt{s} \|\A \vv\|_2 + 2\sigma_s(\x)_{\omega,1}\right).
\]
Using \eqref{l1NSP} once more finally gives
\begin{align}
\|\x-\z\|_{\omega,1} &= \|\vv_S\|_{\omega,1} + \|\vv_{S^c}\|_{\omega,1} \leq (1+\rho) \|\vv_{S^c}\|_{\omega,1} + \tau \sqrt{s} \|\A \vv\|_2\notag\\
& \leq \frac{1+\rho}{1-\rho}\left(\|\z\|_{\omega,1} - \|\x\|_{\omega,1}  + 2\sigma_s(\x)_{\omega,1} \right) + \frac{2 \tau \sqrt{s}}{1-\rho} \|\A (\x-\z)\|_2.\notag
\end{align}
We pass to the proof of \eqref{l2errorbound}. Let $S$ with $\omega(S) \leq s$ be such that 
$\|\vv - \vv_S\|_2 = \|\vv_{S^c}\|_2 = \widetilde{\sigma}_s(\vv)_{\omega,2}$ (recalling that $\|\cdot\|_2 = \|\cdot\|_{\omega,2}$).
Using the weighted Stechkin estimate \eqref{weighted:Stechkin} 
and the robust null space property \eqref{robustNSP} as well as the error bound \eqref{l1errorbound} 
we obtain
\begin{align}
\|\x-\z\|_2 & \leq \|(\x-\z)_{S^c}\|_2 + \|(\x-\z)_{S}\|_2 \notag\\
&\leq \frac{1}{\sqrt{s-\|\w\|_\infty^2}} \|\x-\z\|_{\omega,1} 
+ \frac{\rho}{\sqrt{s}} \|(\x-\z)_{S^c}\|_{\omega,1} + \tau \|\A (\x-\z)\|_2\notag\\
&  \leq \frac{1+ \rho}{\sqrt{s-\|\w\|_\infty^2}} \|\x-\z\|_{\omega,1} + \tau \|\A (\x-\z)\|_2\notag\\
& \leq \frac{2(1+\rho)^2}{(1-\rho)\sqrt{s-\|\w\|_\infty^2}}\left(\|\z\|_{\omega,1} - \|\x\|_{\omega,1} + \sigma_s(\x)_{\omega,1}\right)\notag\\
& + \left(\tau + \frac{2 \tau(1+\rho)\sqrt{s}}{(1-\rho)\sqrt{s-\|\w\|_\infty^2}}\right) \|\A(\x-\z)\|_2.\notag
\end{align}
Since $s \geq 2 \|\w\|_\infty^2$ the statement follows with $C_1 = 2\sqrt{2}(1+\rho)^2/(1-\rho)$ and $C_2 = \tau + 2\sqrt{2}\tau(1+\rho)/(1-\rho)$.
\end{proof}
%
As a consequence of the previous result we obtain error bounds for sparse recovery via weighted $\ell_1$ minimization.
\begin{corollary} 
\label{cor:nsp}
Let $\A \in \C^{m \times N}$ satisfy the weighted robust null space property of order $s$ and constants $\rho \in (0,1)$ and $\tau > 0$. For $\x \in \C^N$ and $\y = \A \x + \vct{e}$ with
$\|\vct{e}\|_2 \leq \eta$, let $\x^\sharp$ be the solution of 
\[
\min_{\z \in \C^\N} \|\z\|_1 \mbox{ subject to } \|\A \z - \y\|_2 \leq \eta.
\]
Then the reconstruction error satisfies
\begin{align}
\|\x-\x^\sharp\|_{\omega,1} & \leq c_1 \sigma_{s}(\x)_{\omega,1} + d_1 \sqrt{s} \eta \label{l1error}\\
\|\x-\x^\sharp\|_2 & \leq c_2 \frac{\sigma_s(\x)_{\omega,1}}{\sqrt{s}} + d_2 \eta,\label{l2error}
\end{align}
where the second bound additionally assumes $s \geq 2 \| \w\|_\infty^2$.
The constants $c_1,c_2,d_1,d_2 > 0$ depend only on $\rho$ and $\tau$.
\end{corollary}
\begin{proof} The reconstruction errors follow from the error bounds in Theorem~\ref{thmNSP} with $\z = \x^{\#}$, noting that $\| \x^{\#} \|_{\omega,1} - \| \x \|_{\omega,1} \leq 0$ and $\| \A(\x - \x^{\#}) \|_2 \leq \| \A \x - \y \|_2 + \| \A \x^{\#} - \y \|_2 \leq 2 \eta$. 

\end{proof}

\begin{remark} \emph{In the case of noiseless measurements, the previous result gives error bounds for equality-constrained weighted $\ell_1$ minimization
by setting $\eta = 0$. Moreover, with a similar technique as used for the previous result, one can generalize \eqref{l1error} and \eqref{l2error}
to error bounds in weighted $\ell_{\omega,p}$ for $1 \leq p \leq 2$, see \cite{fora13} for the unweighted case.}
\end{remark}

\subsection{Weighted restricted isometry property}

It is often unclear how to show the weighted null space property directly for a given matrix. 
In the unweighted case, it therefore has become useful to work instead with the restricted isometry property, which implies the null space property.
As introduced in Definition \ref{def:weighted:RIP}, we define the weighted restricted isometry ($\w$-RIP) constant $\delta_{\omega,s}$ associated to a matrix $\A$ 
as the smallest number such that
\[
(1-\delta_{\omega,s}) \|\x\|_{2}^2 \leq \| \A \x \|_2^2 \leq (1+\delta_{\omega,s}) \| \x \|_2^2 \quad \mbox{ for all } \x \mbox{ with } \|\x\|_{\omega,0} \leq s.
\]
We say that $\A$ satisfies a weighted restricted isometry property ($\w$-RIP) if $\delta_{\omega,s}$ is small for $s$ relatively large compared to $m$. 
 The $\w$-RIP implies the weighted robust null space property and therefore the error bounds \eqref{l1error} and \eqref{l2error} 
 for recovery via weighted $\ell_1$ minimization as shown in the following result.
 
\begin{theorem} 
 \label{ripnsp}
 Let $\A \in \C^{m \times N}$ with 
 $\w$-RIP constant
 \begin{equation}\label{cond:RIP}
 \delta_{\omega,3s} < 1/3
 \end{equation}
 for $s \geq 2 \|\w \|_\infty^2$.
Then $\A$ satisfies the weighted robust null space property of order $s$ with constants $ \rho = 2 \delta_{\omega,3s}/(1-\delta_{\omega,3s}) < 1$ 
and $\tau =  \sqrt{1+\delta_{\omega,3s}}/(1-\delta_{\omega,3s})$.
 \end{theorem}


Before proving Theorem~\ref{ripnsp}, we make the following observations.
As in the unweighted case (see e.g.\ \cite{fora13,ra10}) the $\w$-RIP constants can be rewritten as
\[
\delta_{\omega,s} = \max_{S \subset [N], \omega(S) \leq s} \| \A_S^* \A_S - \text{Id} \|_{2 \to 2},
\] 
{where $\A_S$ denotes the submatrix of $\A$ restricted to the columns indexed by $S$.}
\begin{lemma}\label{lem:RIP:orth} If $\vct{u}, \vct{v} \in \C^{N}$ are such that $\|\vct{u}\|_{\omega,0} \leq s, \|\vct{v}\|_{\omega,0} \leq t$ and $\supp \vct{u} \cap \supp \vct{v} = \emptyset$ then
\[
|\langle \A \vct{u}, \A \vct{v} \rangle| \leq \delta_{\omega,s+t} \|\vct{u}\|_2 \|\vct{v}\|_2.
\]
\end{lemma} 
\begin{proof} Let $S = \supp \vct{u} \cup \supp \vct{v}$ so that $\omega(S) \leq s+t$. Since $\langle \vct{u}, \vct{v}\rangle = 0$ we have
\begin{align}
|\langle \A \vct{u}, \A \vct{v} \rangle| &= |\langle \A_S \vct{u}_S, \A_{S} \vct{v}_S\rangle - \langle \vct{u}_S,\vct{v}_S \rangle |
= |\langle (\A_S^* \A_S - \Id) \vct{u}_S, \vct{v}_S\rangle|\notag\\ 
&\leq \|\A^*_{S} \A_S - \Id \|_{2 \to 2} \|\vct{u}_S\|_2 \|\vct{v}_S\|_2 \leq \delta_{\omega,s+t} \|\vct{u}\|_2 \|\vct{v}\|_2.\notag
\end{align}
This completes the proof.
\end{proof}
Now we are prepared for the proof of the main result of this section.
\begin{proof}[Proof of Theorem~\ref{ripnsp}.] Let $\vct{v} \in \C^N$ and $S\subset [N] $ with $\omega(S) \leq s$. 
We partition $S^c$ into blocks $S_1,S_2,\hdots$ with $s - \| \w\|_{\infty}^2 \leq \omega(S_\ell) \leq s$ according
to the nonincreasing rearrangement of $v_{S^c} \cdot \omega^{-1}_{S^c}$, that is, $|v_j| \omega_j^{-1} \leq |v_k| \omega_k^{-1}$ for all $j \in S_\ell$ and all $k \in S_{\ell-1}$, $\ell \geq 2$.
Then we estimate
\begin{align}
\|\vct{v}_{S} + \vct{v}_{S_1}\|_2^2 & \leq \frac{1}{1-\delta_{\omega,2s}} \|\A(\vct{v}_S + \vct{v}_{S_1})\|_2^2 =
\frac{1}{1-\delta_{\omega,2s}} \left\langle \A (\vct{v}_S + \vct{v}_{S_1}), \A \vct{v} - \sum_{\ell \geq 2} \A \vct{v}_{S_\ell} \right\rangle\notag\\
& = \frac{1}{1-\delta_{\omega,2s}} \left( \left\langle \A (\vct{v}_S + \vct{v}_{S_1}), \A \vct{v}\right\rangle - \sum_{\ell \geq 2} 
\left\langle \A (\vct{v}_S + \vct{v}_{S_1}), \A \vct{v}_{S_\ell}\right\rangle\right)\notag\\
& \leq \frac{1}{1-\delta_{\omega,2s}} \left(\sqrt{1+\delta_{\omega,2s}} \|\vct{v}_S + \vct{v}_{S_1}\|_2  \|\A \vct{v}\|_2+ \delta_{\omega,3s}\|\vct{v}_S + \vct{v}_{S_1}\|_2 \sum_{\ell \geq 2}  \|\vct{v}_{S_\ell}\|_2 \right),\notag
\end{align}
where we have used Lemma \ref{lem:RIP:orth} in the third line.
Dividing by $\|\vct{v}_{S} + \vct{v}_{S_1}\|_2$ and using the fact that $\delta_{\omega,2s} \leq \delta_{\omega, 3s}$ we arrive at 
\[
\|\vct{v}_S\|_2 \leq \|\vct{v}_{S} + \vct{v}_{S_1}\|_2 \leq  \frac{\delta_{\omega,3s}}{1-\delta_{\omega,3s}} \sum_{\ell \geq 2} \|\vct{v}_{S_\ell}\|_2 +  \frac{\sqrt{1+\delta_{\omega,3s}}}{1-\delta_{\omega,3s}} \|\A \vct{v}\|_2.
\]
Now for $k \in S_\ell$, set $\alpha_k = (\sum_{j \in S_\ell} \omega_j^2)^{-1} \omega_k^2 \leq (s-\|\w\|_\infty^2)^{-1} \omega_k^2$. Then $\sum_{k \in S_{\ell}} \alpha_k = 1$ and
$|v_j| \omega_j^{-1} \leq \sum_{k \in S_{\ell-1}} \alpha_k |v_k| \omega_k^{-1} \leq (s - \|\w\|_\infty^2)^{-1} \sum_{k \in S_{\ell-1}} |v_k| \omega_k$ 
for all $j \in S_\ell$, $\ell \geq 2$, by our construction of the partitioning. By the Cauchy-Schwarz inequality and since $s \geq 2\|\w\|_\infty^2$ this gives
\[
\|\vct{v}_{S_\ell}\|_2 \leq \frac{\sqrt{s}}{s-\|\w \|_2^2} \|\vct{v}_{S_{\ell-1}}\|_{\omega,1} \leq \frac{2}{\sqrt{s}} \|\vct{v}_{S_{\ell-1}}\|_{\omega,1}.
\]
Therefore,
\begin{align}
\|\vct{v}_S\|_2 & \leq \frac{2 \delta_{\omega,3s}}{(1-\delta_{\omega,3s})\sqrt{s}} \sum_{\ell\geq 1} \|\vct{v}_{S_\ell}\|_{\omega,1} + \frac{\sqrt{1+\delta_{\omega,3s}}}{1-\delta_{\omega,3s}} \|\A \vct{v}\|_2\notag\\
& \leq \frac{2 \delta_{\omega,3s}}{(1-\delta_{\omega,3s})\sqrt{s}} \|\vct{v}_{S^c}\|_{\omega,1} +  \frac{\sqrt{1+\delta_{\omega,3s}}}{1-\delta_{\omega,3s}} \|\A \vct{v}\|_2.\notag
\end{align}
This yields the desired estimate with $\tau = \sqrt{1+\delta_{\omega,3s}}/(1-\delta_{\omega,3s})$ and $\rho = 2 \delta_{\omega,3s}/(1-\delta_{\omega,3s})$ which is strictly smaller than
$1$ if $\delta_{\omega,3s} < 1/3$.
\end{proof}
We remark that we did not attempt to provide the optimal constant in \eqref{cond:RIP}. Improvements can be achieved by pursuing more complicated arguments, see e.g.~\cite{fora13}. 
Also, conditions involving $\delta_{\omega,2s}$ instead of $\delta_{\omega,3s}$ are possible.
 
\section{Weighted RIP estimates for orthonormal systems}

In this section, we provide a quite general class of structured random matrices which satisfy the $\w$-RIP.
  Precisely,  the main theorem of this section is that matrices arising from orthonormal systems satisfy the $\w$-RIP as long as the weights grow at least as quickly as the $L_{\infty}$ norms of the functions they correspond to.  This extends existing unweighted RIP results for finite \emph{bounded} orthonormal systems $(\psi_j)_{j \in \Lambda}$ such that $\sup_{j \in \Lambda} \|\psi_j\|_\infty \leq K$ for some constant $K \geq 1$.  For bounded orthonormal systems, the following unweighted RIP estimates have been shown. 

\begin{proposition}[Theorems 4.4 and 8.4, \cite{ra10}]\label{prop:classical}
\label{thm:BOS}
Fix parameters $\delta, \gamma \in (0,1)$.  Let $(\psi_j)_{j \in \Lambda}$ be a bounded orthonormal system 
with uniform bound $K$.
Suppose
\begin{align}
m &\geq CK^2  \delta^{-2} s \log^2(s) \log(m) \log(N), \nonumber \\
m &\geq DK^2 \delta^{-2} s \log(1/\gamma),
\end{align}
where $N = |\Lambda|$.
Assume that $t_1, t_2, \dots, t_m$ are drawn independently from the orthogonalization measure $\nu$
associated to the orthonormal system.  
Then  with probability exceeding $1 - \gamma$, the normalized sampling matrix $\tilde{\A} \in \C^{m \times N}$ with entries 
$\tilde{A}_{\ell,k} = \frac{1}{\sqrt{m}}\psi_k(t_\ell), \quad \ell \in [m], k \in [N],$
satisfies the restricted isometry property of order $s$, that is, $\delta_{s} \leq \delta$. 
\end{proposition}

In fact, we can allow $\|\psi_j\|_\infty$ to depend on $j$ if we ask only for $\w$-RIP with weights $\omega_j =  \|\psi_j\|_\infty$, or more generally, $\omega_j \geq \|\psi_j\|_\infty$.  This is the content of the following theorem.

\vspace{3mm}

\begin{theorem}[$\w$-RIP for orthonormal systems]
\label{mainthm}
Fix parameters $\delta, \gamma \in (0,1)$.  Let $(\psi_j)_{j \in \Lambda}$ be an orthonormal system of finite size $N = | \Lambda |$.
Consider weights satisfying $\omega_j \geq \|\psi_j\|_\infty$. 
Fix
$$
m \geq C \delta^{-2} s \max \{\log^3(s) \log(N), \log(1/\gamma)\}
$$
and suppose that $t_1, t_2, \dots, t_m$ are drawn independently from the orthogonalization measure
associated to the $(\psi_j)$. 
Then  with probability exceeding $1 - \gamma$, the normalized sampling matrix $\tilde{\A} \in \C^{m \times N}$ with entries $
\tilde{A}_{\ell,k} = \frac{1}{\sqrt{m}}\psi_k(t_\ell) $
satisfies the weighted restricted isometry property of order $s$, that is, $\delta_{\omega,s} \leq \delta$. 
\end{theorem}

We remark that if $K = \max_j \| \psi_j \|_{\infty}$ is a constant independent or only mildly dependent on $N$, then Theorem \ref{mainthm} essentially reduces to Proposition \ref{thm:BOS}. 
Note, however, that in the restricted parameter regime of $s \lesssim \log(N)$, the above result gives a slight improvement
over the classical result stated in Proposition \ref{prop:classical} -- generalizing the main result of \cite{chguve12}.  The remainder of this section is reserved for the proof of Theorem \ref{mainthm}.

\begin{proof}[Proof of Theorem~\ref{mainthm}.]
The proof proceeds in a similar manner to those in \cite{ruve08,ra10,fora13}, with some adaptations to account for the weights -- see the application of {Maurey's lemma (Lemma \ref{lem:Maurey})} -- and with a twist from \cite{chguve12}
leading to the slight improvement in the logarithmic factor.  Note that our analysis improves the result of  \cite{chguve12} in terms of the probability estimate.

Introducing the set
\begin{equation}
\label{Ts}
T^{s,N}_{\omega} = \{ \x \in \C^N, \| \x \|_2 \leq 1, \| \x \|_{0,\omega} \leq s \},  
\end{equation}
we can rephrase the weighted isometry constant of $\A$ as
\[
\delta_{\omega,s} = \sup_{\x \in T^{s,N}_{\omega}} | \langle (\A^*\A - \Id) \x, \x \rangle|.
\]
The quantity
\begin{equation}
\label{matrixnorm}
\vertiii{\BB}_s := \text{sup}_{\z \in T^{s,N}_\omega} \left| \scalprod{\BB\z}{\z} \right|
\end{equation}
defines a semi-norm on matrices $\BB \in \C^{N \times N}$, and we can write 
$$
\delta_{\omega,s} = \vertiii{\A^*\A - \Id}_s.
$$
Consider the random variable associated to a column of the adjoint matrix,
\begin{equation}
\label{zl}
\X_{\ell} = \Big(\overline{\psi_j(t_{\ell})} \Big)_{j \in \Lambda} .
\end{equation}
By orthonormality of the system $(\psi_j)$ we have  $\E \X_{\ell} \X_{\ell}^* = \Id$, and the {restricted isometry} constant equals
$$
\delta_{\omega,s} = \vertiii{ \frac{1}{m} \sum_{\ell=1}^m \X_{\ell} \X_{\ell}^* - \Id}_s = \frac{1}{m} \vertiii{ \sum_{\ell=1}^m (\X_{\ell} \X_{\ell}^* - \E \X_{\ell} \X_{\ell}^*)}_s.
$$
As a first step we estimate the expectation of $\delta_{\omega,s}$ and later use a concentration result to deduce the probability estimate.
We introduce a Rademacher sequence $\veps = (\epsilon_1,\hdots,\epsilon_m)$, i.e., a sequence of independent Rademacher variables $\epsilon_\ell$ taking the values $+1$ and $-1$ with equal probability, {also independent of the variables $\X_\ell$}.
Symmetrization, see e.g.\ \cite[Lemma 6.3]{leta91} or \cite[Lemma 6.7]{ra10}, yields
\begin{align}
\E \delta_{\omega,s} 
& \leq \frac{2}{m} 
\E \vertiii{ \sum_{\ell=1}^m \epsilon_{\ell} \X_{\ell} \X_{\ell}^*}_s = \frac{2}{m} \E_{\X} \E_{\veps} \sup_{\x \in T_{\omega}^{s,N}} |\langle \sum_{\ell=1}^m \veps_\ell \X_\ell \X_\ell^* \x, \x\rangle| \notag\\
& = \frac{2}{m} \E_{\X} \E_\epsilon \sup_{\x \in T_{\omega}^{s,N}} |\sum_{\ell=1}^m \epsilon_\ell |\langle \X_\ell, \x\rangle|^2 |.\label{Edelta:bound}
\end{align}
Conditional on $(\X_\ell)$, we arrive at a Rademacher (in particular, subgaussian) process indexed by $T_{\omega}^{s,N}$. For a set $T$, a metric $d$
and given $u > 0$, the covering numbers ${\cal N}(T,d,u)$ are defined as the smallest number of balls with respect to $d$ and centered at points of $T$
necessary to cover $T$. For fixed $(\X_\ell)$, we work with the
(pseudo-)metric
\[
d(\x,\z) = \left( \sum_{\ell=1}^m (|\langle \X_\ell, \x\rangle |^2 - |\langle \X_\ell, \z\rangle|^2)^2\right)^{1/2}.
\]
Then Dudley's inequality \cite{li12,leta91,ra10,fora13} implies that 
\[
\E_{\veps} \sup_{\x \in T_{\omega}^{s,N}} |\langle \sum_{\ell=1}^m \epsilon_\ell \X_\ell \X_\ell^* \x, \x\rangle|
\leq 4\sqrt{2} \int_0^{\infty} \sqrt{\log({\cal N}(T_{\omega}^{s,N},d,u))} du.
\]
In order to continue we estimate the metric $d$ using H{\"o}lder's inequality with exponents $p \geq 1$ and $q \geq 1$ satisfying $1/p + 1/q = 1$ to be specified later on. For $\x,\z \in T_{\omega}^{s,N}$, this gives
\begin{align}
d(\x,\z) & =  \left( \sum_{\ell=1}^m (|\langle \X_\ell, \x\rangle | + |\langle \X_\ell, \z\rangle|)^2(|\langle \X_\ell,\x\rangle| - |\langle \X_\ell, \z\rangle|)^2\right)^{1/2}\notag\\
& \leq \left(\sum_{\ell=1}^m (|\langle \X_\ell, \x\rangle| + |\langle \X_\ell, \z\rangle|)^{2p}\right)^{1/(2p)}  \left(\sum_{\ell=1}^m |\langle \X_\ell, \x - \z \rangle|^{2q}\right)^{1/(2q)}\notag\\
& \leq 2 \sup_{\x \in T_\omega^{s,N}} \left(\sum_{\ell=1}^m |\langle \X_\ell, \x\rangle|^{2p}\right)^{1/(2p)} \left(\sum_{\ell=1}^m |\langle \X_\ell, \x - \z \rangle|^{2q}\right)^{1/(2q)}. 
\end{align}
In the standard analysis \cite{ruve08,ra10,fora13}, 
this bound is applied for $p=1$, $q=\infty$. Following \cite{chguve12}, we will achieve a slightly better logarithmic factor 
by working with a {different value of $p$ to be determined later.} 

For any realization of $(\X_\ell)$, we have $|(\X_\ell)_j| \leq \|\psi_j\|_\infty \leq \omega_j$ by assumption. For $\x \in T_{\omega}^{s,N}$
with $S = \supp \x$ we have $\sum_{j \in S} \omega_j^2 \leq s$, resulting in
\begin{equation}\label{bound:inner}
|\langle \X_\ell, \x\rangle| \leq \sum_{j \in S} \omega_j |x_j| \leq (\sum_{j \in S} \omega_j^2)^{1/2} \|\x\|_2 \leq \sqrt{s}.
\end{equation}
This gives
\begin{align}
\sup_{\x \in T_\omega^{s,N}} \left(\sum_{\ell=1}^m |\langle \X_\ell, \x\rangle|^{2p}\right)^{1/(2p)} 
& = \sup_{\x \in T_\omega^{s,N}} \left(\sum_{\ell=1}^m |\langle \X_\ell, \x\rangle|^{2}|\langle \X_\ell, \x\rangle|^{2(p-1)} \right)^{1/(2p)} \notag\\
& \leq s^{(p-1)/(2p)} \left(\sup_{\x \in T_\omega^{s,N}} \sum_{\ell=1}^m |\langle \X_\ell, \x\rangle|^{2} \right)^{1/(2p)}. \notag
\end{align}
Introducing the (semi-)norm
\[
\| \x \|_{X,q} = \left(\sum_{\ell=1}^m |\langle \X_\ell, \x\rangle|^{2q}\right)^{1/(2q)}
\]
and using basic properties of covering numbers, we obtain
\begin{align}
&\E_{\veps} \sup_{\x \in T_{\omega}^{s,N}} |\langle \sum_{\ell=1}^m \epsilon_\ell \X_\ell \X_\ell^* \x, \x\rangle|\notag\\
&\leq C_1 s^{(p-1)/(2p)} \left(\sup_{\x \in T_\omega^{s,N}} \sum_{\ell=1}^m |\langle \X_\ell, \x\rangle|^{2} \right)^{1/(2p)}
\int_0^\infty \sqrt{\log({\cal N}(T_\omega^{s,N},\|\cdot\|_{\X,q},u))} du, \label{bound:Eveps1}
\end{align}
where $C_1$ is a suitable constant.
Next, we estimate the covering numbers appearing above in two different ways.

Let us first derive a bound which is good for small values of $u$. 
It follows from \eqref{bound:inner} that, for $\x \in T_\omega^{s,N}$,
\begin{equation}\label{Xql2:bound}
\| \x \|_{\X,q} \leq \left(\sum_{\ell=1}^m (\sqrt{s} \|\x\|_2)^{2q}\right)^{1/2q} = \sqrt{s} m^{1/(2q)} \|\x\|_2.
\end{equation}
Denoting $B_S$ to be the $\ell_2$ unit ball of vectors with support in $S$ and applying the volumetric covering number bound (see e.g.~\cite[Proposition 10.1]{ra10}) gives
\begin{align}
{\cal N}(T_\omega^{s,N},\|\cdot\|_{\X,q},u) & \leq \sum_{S \subset \Lambda : \omega(S) \leq s} {\cal N}(B_S, \sqrt{s} m^{1/(2q)} \|\cdot\|_2,u)\notag\\
& \leq \left(\begin{matrix} N \\ s \end{matrix}\right) \left(1+ \frac{2 \sqrt{s} m^{1/(2q)}}{u}\right)^{2s} \leq
(eN/s)^s \left(1+ \frac{2 \sqrt{s} m^{1/(2q)}}{u}\right)^{2s}, \notag
\end{align}
where we have applied \cite[Proposition C.3]{fora13} (see also \cite[p.~72]{ra10}) in the last step.

We use Maurey's lemma \cite{ca85} -- see also \cite[Lemma 4.2]{krmera12} for the precise form below -- in order to deduce a covering number bound which is good for larger values of $u$.
Below, $\conv(U)$ denotes the convex hull of a set $U$.

\begin{lemma}\label{lem:Maurey} For a normed space $X$, consider a finite set ${\cal U} \subset X$ of cardinality $N$, and assume that for
every $L \in \N$ and $(\vct{u}_1,\hdots,\vct{u}_L) \in {\cal U}^L$, $\E_{\veps} \|\sum_{j=1}^L \epsilon_j \vct{u}_j \|_X \leq A\sqrt{L}$,
where $\veps$ denotes a Rademacher vector. Then for every $u>0$,
$$
\log {\cal N}({\conv}({\cal U}), \| \cdot \|_X, u) \leq c (A/u)^2 \log N.
$$
The constant $c > 0$ is universal.
\end{lemma} 

To apply this lemma, we first observe that $T_{\omega}^{s,N} \subset \sqrt{2s} \conv(U)$, where 
\[
U = \{ \pm \omega_j^{-1} \vct{e}_j, \pm i \omega_j^{-1} \vct{e}_j, j \in \Lambda\}.
\]
Here, $\vct{e}_j$ denotes the $j$-th canonical unit vector. For a Rademacher vector $\veps = (\epsilon_1,\hdots,\epsilon_L)$, and $\vct{u}_1,\hdots,\vct{u}_L \in U$ we have
\begin{align}
& \E_{\veps} \|\sum_{j=1}^L \epsilon_j \vct{u}_j \|_{X,q}  \leq \left( \E \|\sum_{j=1}^L \epsilon_j \vct{u}_j \|_{X,q}^{2q}\right)^{1/(2q)}
= \left( \E \sum_{\ell=1}^m |\langle \X_\ell, \sum_{j=1}^L \epsilon_j \vct{u}_j\rangle|^{2q} \right)^{1/(2q)}\notag\\
& = \left( \sum_{\ell=1}^m \E  \left| \sum_{j=1}^L \epsilon_j \langle \X_\ell, \vct{u}_j \rangle\right|^{2q} \right)^{1/(2q)} 
 \leq 2 e^{-1/2} \sqrt{2q} \left(\sum_{\ell=1}^m \| {(\langle \X_\ell, \vct{u}_j\rangle)}_{j=1}^L \|_2^{2q}\right)^{1/(2q)}.\notag
\end{align}
In the last step, we have applied Khintchine's inequality, see e.g.~\cite[Corollary 6.9]{ra10}. Using that $|(\X_\ell)_k| \leq \|\psi_k\|_\infty \leq \omega_k$, we have, 
for any vector $\vct{u}_j \in U$, say $\vct{u}_j = \omega_k^{-1} \vct{e}_k$, that
\[
|\langle \X_\ell, \vct{u}_j\rangle| = |\omega_k^{-1} (\X_\ell)_k| \leq 1.
\]
Therefore, $\|{(\langle \X_\ell, \vct{u}_j\rangle)}_{j=1}^L \|_2 \leq \sqrt{L}$ for any $L$ and
\[
 \E_{\veps} \|\sum_{j=1}^L \epsilon_j \vct{u}_j \|_{\X,q}  \leq 2 e^{-1/2} \sqrt{2q} m^{1/(2q)} \sqrt{L}.
\]
An application of Lemma~\ref{lem:Maurey} with $A = 2 e^{-1/2} \sqrt{2q} m^{1/(2q)}$ yields
\[
\sqrt{\log {\cal N}(T_\omega^{s,N},\|\cdot\|_{\X,q},u)} \leq  \sqrt{\log {\cal N}({\conv(U)}, \| \cdot \|_{\X,q}, u/\sqrt{2s} )} \leq C_2 \sqrt{q m^{1/q} s \log(4 N)}\, u^{-1}
\]
with $C_2 = 4 e^{-1/2} \sqrt{c}$.

Observe that it is enough to choose the upper integration bound in the Dudley type integral as $\sqrt{s} m^{1/(2q)}$ because for $u > \sqrt{s} m^{1/(2q)}$ we have
${\cal N}(T_\omega^{s,N},\|\cdot\|_{X,q}, u) = 1$ by \eqref{Xql2:bound}. Splitting then the Dudley integral into two parts and using the appropriate bounds for the covering
numbers, we obtain, for $\kappa \in \big( 0,\sqrt{s}m^{1/(2q)} \big)$,
\begin{align}
&\int_0^\infty \sqrt{\log({\cal N}(T_\omega^{s,N},\|\cdot\|_{\X,q},u))} du\notag\\
& \leq \int_0^\kappa \sqrt{s \log(eN/s) + 2s \log(1+ 2 \sqrt{s}m^{1/(2q)}/u)} du\notag\\
& + C_2 \sqrt{q m^{1/q} s \log(4N)} \int_{\kappa}^{\sqrt{s}m^{1/(2q)}} u^{-1} du \notag\\
& \leq \kappa \sqrt{s \log(eN/s)} + \sqrt{2s} \kappa \sqrt{\log(e(1+\sqrt{s}m^{1/(2q)}/\kappa))} \notag\\
& + C'\sqrt{q m^{1/q} s \log(4N)} \log(\sqrt{s} m^{1/(2q)}/ \kappa). \notag
\end{align}
In the last step, we have applied \cite[Lemma 10.3]{ra10}. Choosing $\kappa = m^{1/(2q)}$ yields
\[
\int_0^\infty \sqrt{\log({\cal N}(T_\omega^{s,N},\|\cdot\|_{\X,q},u))} du \leq C_3 \sqrt{q s m^{1/q} \log(N) \log^2(s)}. 
\]
A combination with \eqref{bound:Eveps1} and \eqref{Edelta:bound} gives
\begin{align}
\E \delta_{\omega,s} & \leq \frac{C_3 s^{(p-1)/(2p)}\sqrt{q m^{1/q} s \log(N) \log^2(s)}}{m}   \E \sup_{\x \in T_{\omega}^{s,N}} \left( \sum_{\ell=1}^m |\langle \X_\ell,x\rangle|^2 \right)^{1/(2p)}\notag\\
&\leq \frac{C_3s^{1/2+(p-1)/(2p)}\sqrt{q \log(N) \log^2(s)}}{m^{1-1/(2q)}m^{-1/(2p)}} \E \left( \frac{1}{m} \vertiii{ \sum_{\ell=1}^m \X_{\ell} \X_{\ell}^* - \Id}_s +  \vertiii{\Id}_s\right)^{1/(2p)} \notag\\
&\leq  \frac{C_3s^{1/2+(p-1)/(2p)}\sqrt{q \log(N) \log^2(s)}}{m^{1/2}} \sqrt{\E \delta_{\omega,s} + 1}.\notag
\end{align}
Hereby, we have applied H{\"o}lder's inequality and used that $1/q+1/p = 1$ as well as $p \geq 1$.
Choosing $p = 1+ 1/\log(s)$ and $q=1+\log(s)$ gives $s^{(p-1)/(2p)} \leq s^{(p-1)/2} = s^{1/(2 \log(s))} = e^{1/2}$ and
\[
\E \delta_{\omega,s} \leq C_4 \sqrt{\frac{s \log(N) \log^3(s)}{m}} \sqrt{\E \delta_{\omega,s} + 1}.
\]
Completing squares finally shows that 
\begin{align}\label{bound:expect}
\E \delta_{\omega,s} \leq C_5 \sqrt{\frac{s \log(N) \log^3(s)}{m}}
\end{align}
provided the term under the square root is bounded by $1$.

\medskip

For the probability bound, we show that $\delta_{\omega,s}$ does not deviate much from its expectation. 
By \eqref{bound:expect}, $\E \delta_{\omega,s} \leq \delta/2$ for some $\delta \in (0,1)$ if
\begin{align}\label{m:bound:expect}
m \geq C_6 \delta^{-2} s \log^3(s) \log(N)
\end{align}
with $C_6 = 4 C_5^2$. Similarly to \cite[Section 8.6]{ra10} we write
\[
\delta_{\omega,s} = \frac{1}{m} \sup_{(\z,\vct{w}) \in Q_{\omega,*}^{s,N}} \operatorname{Re} \left\langle \sum_{\ell=1}^m (\X_\ell \X_\ell^* - \Id) \z, \w \right\rangle
\]
where $Q_{\omega,*}^{s,N}$ denotes a dense countable subset of 
\[
Q_{\omega}^{s,N} = \bigcup_{S \subset \Lambda, \omega(S) \leq s} Q_{S}, \qquad Q_S = \{ (\z,\vct{w}) : \|\z\|_2 = \|\vct{w}\|_2 = 1, \supp \z, \supp \vct{w} \subset S\}.
\]
With the functions $f_{\z,\vct{w}}(X) = \operatorname{Re} \langle (\X \X^{*}-\Id) \x, \vct{w} \rangle$ we can write $\delta_{\omega,s}$ as the supremum of an empirical process
\[
\delta_{\omega,s} = \frac{1}{m} \sup_{(\z,\vct{w}) \in Q_{\omega,*}^{s,N}} \sum_{\ell=1}^m f_{\z,\vct{w}}(\X_\ell).
\]
Since $\E \X_\ell \X_\ell^* = \Id$ we have $ \E f_{\z,\vct{w}}(\X_\ell) = 0$ for all $\z,\vct{w}$. Further,
for $(\x,\vct{w}) \in Q_{S}$ with $\omega(S) \leq s$ and for any realization of $\X_\ell$, we have
\[
|f_{\z,\vct{w}}(\X_\ell)| \leq \max\{1, \max_{\substack{\x: \supp \x \subset S\\ \|\x\|_2 = 1}} |\langle \X_\ell \X_\ell^* \x, \x\rangle| \}
= \max\{1, \max_{\substack{\x: \supp \x \subset S\\ \|\x\|_2 = 1}} |\langle \X_\ell, \x\rangle|^2\} \leq s.
\]
Moreover, 
\begin{align}
\E |f_{\z,\vct{w}}(\X_\ell)|^2 & = \E |\langle (\X_\ell \X_\ell^* - \Id) \z, \vct{w}\rangle|^2\notag\\
& = \E |\langle \X_\ell \X_\ell^* \z, \vct{w}\rangle|^2 
- 2 \Re (\E [\langle \X_\ell \X_\ell^* \z,\vct{w}\rangle]  \langle \vct{w},\vct{z}\rangle) + |\langle \z, \vct{w}\rangle|^2 \notag\\
& = \E [ |\langle \X_\ell, \z\rangle|^2 |\langle \X_\ell, \vct{w}\rangle|^2] - |\langle \z, \vct{w}\rangle|^2
\leq s \E |\langle \X_\ell, \z \rangle|^2 = s.\notag
\end{align}
With these bounds for $f_{\z,\vct{w}}(\X_\ell)$ together with \eqref{m:bound:expect}, 
the Bernstein inequality for the supremum of an empirical process, see e.g.~\cite[Theorem 6.25]{ra10}, \cite{bo02-1} or
\cite[Theorem 8.42]{fora13}, yields, for $\delta \in (0,1)$,
\begin{align}
\P(\delta_{\omega,s} \geq \delta)& \leq \P(\delta_{\omega,s} \geq \E \delta_{\omega,s} + \delta/2) \notag\\
&= \P(\sup_{(\z,\vct{w}) \in Q_{\omega,*}^{s,N}} \sum_{\ell=1}^m f_{\z,\vct{w}}(\X_\ell) \geq \E \sup_{(\z,\vct{w}) \in Q_{\omega,*}^{s,N}} \sum_{\ell=1}^m f_{\z,\vct{w}}(\X_\ell) + \delta m /2)\notag\\
&\leq \exp\left(- \frac{(\delta m/2)^2/2}{m s + 2 s (\delta m/2) + (\delta m/2) s/3}  \right)
\leq \exp\left(- \frac{\delta^2 m}{C_7 s} \right),
\end{align}
where $C_7=8(1+2+1/6) \leq 26$. The last term is bounded by $\gamma \in (0,1)$ if $m \geq C_7 \delta^{-2} s \log(1/\gamma)$.
Altogether we have $\delta_{\omega,s} \leq \delta$ with probability at least $1-\gamma$ if
\[
m \geq C_8 \delta^{-2} s \max\{\log^3(s) \log(N), \log(1/\gamma) \},
\]
where $C_8 = \max\{C_6, C_7\}$.
This completes the proof.
\end{proof}

{\section{Main interpolation estimates}\label{alltogether}}

Using the concepts of weighted null space and weighted restricted isometry property, and together with Theorem~\ref{mainthm} on the $\omega$-RIP for orthonormal systems, we now prove Theorems~\ref{thm:intro1} and \ref{thm:infRIPprob} concerning interpolation via weighted $\ell_1$ minimization, and a more general result for functions with coefficients in weighted $\ell_p$ spaces.  We first state a finite-dimensional result which allows for noisy measurements. 

\begin{theorem}\label{thm:interpolation2}
Suppose $(\psi_j)_{j \in \Lambda}$ is an orthonormal system with $| \Lambda | = N$ finite.  Consider weights $\omega_j \geq \|\psi_j\|_\infty$.
{For $s \geq 2  \max_j | \omega_j |^2$ and $\gamma \in (0,1)$, fix} a number of samples 
\begin{equation}
\label{bound:RIP:m}
m \geq c_0 s \max\{\log^3(s) \log(N), \log(1/\gamma)\}.
\end{equation}
Suppose that $t_\ell$, $\ell=1,\hdots,m$, are drawn independently from the orthogonalization measure associated to the $(\psi_j)$.  Let $\A \in \C^{m \times N}$ be the sampling matrix with entries $A_{\ell,k} = \psi_k(t_\ell).$
Then with probability exceeding $1 - \gamma$, the following holds for all 
functions $f  = \sum_{j \in \Lambda} x_j \psi_j$.  Given noisy samples $y_\ell = f(t_\ell) + \xi_{\ell}$, $\ell=1,\hdots,m$, with $\| {\bf \xi} \|_2 \leq \eta$, let $\x^\sharp$ be the solution of
\[
\min \|\z\|_{\omega,1} \mbox{ subject to }  \| \A \z - \y \|_2 \leq  \eta
\]
and set $f^\sharp(t) = \sum_{j \in \Lambda} x_j^\sharp \psi_j(t)$. Then
\begin{align}
\| f - f^\sharp\|_{L_\infty} \leq \triple f - f^\sharp \triple_{\omega,1} & \leq c_1 \sigma_s(f)_{\omega,1} + d_1  \eta \sqrt{\frac{s}{m}}, \notag\\
\| f - f^\sharp \|_{L_2} &  \leq c_2 \frac{ \sigma_s(f)_{\omega,1}}{\sqrt{s}} + d_2  \eta \sqrt{\frac{1}{m}} \notag.\notag
\end{align}
Above, $c_0, c_1, d_1, c_2,$ and $d_2$ are {universal} constants. 
\end{theorem}

\begin{proof}
By Theorem \ref{mainthm}, the normalized sampling matrix $\tilde{\A} = \frac{1}{\sqrt{m}} \A \in \C^{m \times N}$ satisfies the $\w$-RIP of order $s$ and constant $\delta_{\omega,s} \leq 1/3$ with probability exceeding $1 - \gamma$, at the stated number of measurements in \eqref{bound:RIP:m}.  Given that $\tilde{\A}$ satisfies the $\w$-RIP, Theorem \ref{ripnsp} implies that $\tilde{\A}$ satisfies the weighted null space property with constants $ \rho  < 1$ 
and $\tau  >0$.  
The bounds for $\triple f - f^\sharp \triple_{\omega,1}$ and $\| f - f^\sharp \|_{L_2}$ follow from Corollary~\ref{cor:nsp}. 
In order to obtain the bound on $\| f - f^\sharp\|_{L_\infty},$ recall that $\| \psi_j \|_{\infty} \leq \omega_j$ by assumption, so that the 
reconstruction error in  $\ell_{\omega,1}$ implies 
\begin{eqnarray}
\| f - f^{\sharp} \|_{\infty} &\leq& \sum_{j=1}^N | x_j - x_j^{\sharp} | \| \psi_j \|_{\infty} \leq \| \x - \x^{\sharp} \|_{\omega,1} = \triple f - f^{\sharp} \triple_{\omega,1}.\nonumber
\end{eqnarray}
\end{proof}

\begin{remark}
\emph{
Theorem \ref{thm:intro1} in the introduction corresponds to the special case of Theorem \ref{thm:interpolation2} when there is no noise, $\eta = 0$, and with $\gamma = N^{-\log^3(s)}$ chosen to balance both terms in the maximum in \eqref{bound:RIP:m} so that
$m \geq c_0 s \log^3(s) \log(N)$ implies the stated error bounds with probability at least $1- N^{-\log^3(s)}$. 
}
\end{remark}

\subsection{Proof of Theorem  \ref{thm:infRIPprob}}

If the index set $\Lambda$ is countably infinite then we first have to restrict to a suitable finite subset $\Lambda_0$ before applying {weighted} $\ell_1$ minimization for reconstruction. The suitable finite subset we consider is {$\Lambda_0 = \{j: \omega_j^2 \leq s/2\}$}.  
The basic idea is to treat the samples of $f = \sum_{j \in \Lambda} x_j \psi_j$ as perturbed or noisy samples of the finite-dimensional approximation $f_0 = \sum_{j \in \Lambda_0} x_j \psi_j$, splitting $\Lambda$ into $\Lambda_0$ and $\Lambda_R$, decomposing $f$ as $f = f_0 + f_R$,
and treating $f_R(t_j) = \sum_{j \in \Lambda_R} x_j \psi_j(t_j)$ as noise on the observed samples in hopes of applying Theorem~\ref{thm:interpolation2}.

The remainder of the proof consists in showing that the error 
$\sum_{\ell = 1}^m | f_R(t_{\ell}) |^2 = \eta^2$ is small with high probability.
Since the sampling points $t_1, t_2, \dots, t_m$ are drawn i.i.d. from the orthogonalization measure associated 
to $( \psi_j )$, the random variables $| f_R(t_{\ell}) |^2$ are independent and identically distributed, with 
\begin{equation}\label{Bernst:bound1}
 \E \Big(| f_R(t_{\ell}) |^2 \Big) = \sum_{j \in \Lambda_R} |x_j|^2.
\end{equation}
{Since $\omega_j^2 \geq s/2$ for $j \in \Lambda_R$ by construction, we have
\[
 \sum_{j \in \Lambda_R} |x_j |^2 \leq  \frac{2}{s} \sum_{j \in \Lambda_R} | x_j|^2 \omega_j^2
 \leq \frac{2}{s} \Big( \sum_{j \in \Lambda_R} | x_j |  \omega_j \Big)^2 = \frac{2}{s} \| f_R \|^2_{\omega, 1}.
\]
We further have the sup-norm bound 
$| f_R(t_{\ell}) | \leq \| f_R \|_{L_\infty} \leq  \sum_{j \in \Lambda_R} | x_j | \omega_j =  \| f_R \|_{\omega, 1}.$ Therefore, the variance
of the mean-zero variable $|f_R(t_{\ell}) |^2 - \E \Big(| f_R(t_{\ell}) |^2 \Big)$ is bounded by
\[
 \E \Big( |f_R(t_{\ell}) |^2 - \E \Big(| f_R(t_{\ell}) |^2\Big)\Big)^2 \leq  \E \Big(| f_R(t_{\ell}) |^4 \Big) \leq
 \|f_R \|_{\omega,1}^2 \E \Big(| f_R(t_{\ell}) |^2 \Big) \leq \frac{2}{s} \|f_R\|_{\omega,1}^4. 
\] 
We now apply Bernstein's inequality to obtain the probability bound
\[
\mathbb{P} \left\{\left|  \frac{1}{m} \sum_{\ell = 1}^m |f_R(t_{\ell}) |^2-  \sum_{j \in \Lambda_R} x_j^2 \right| \geq \kappa \right\} \leq \exp \left\{ - \frac{m \kappa^2/2}{2 \|f_R\|_{\omega,1}^4/s +   \kappa \| f_R \|_{\omega, 1}^2/3} \right\}.
\]
Setting $\kappa =  \frac{3}{s} \| f_R \|^2_{\omega, 1}$ in Bernstein's inequality gives
$$
\mathbb{P} \left\{\left|  \frac{1}{m} \sum_{\ell = 1}^m | f_R(t_{\ell})|^2 -  \sum_{j \in \Lambda_R} x_j^2 \right| \geq \frac{3}{s} \| f_R \|^2_{\omega, 1}\right\} \leq \exp \left\{ - \frac{3m}{ 2s} \right\}.
$$
For the number of measurements $m = c_0 s \log^3(s) \log(N)$ stated in Theorem \ref{thm:infRIPprob},
we therefore have by \eqref{Bernst:bound1}
$$
\mathbb{P} \left\{ \frac{1}{m}  \sum_{\ell = 1}^m | f_R(t_{\ell}) |^2 \geq \frac{4}{s} \| f_R \|^2_{\omega, 1}\right\} \leq N^{-\log^3(s)}.$$}%
Note that $\sigma_s(f)_{\omega,1} = \sigma_s(f_0)_{\omega,1} + \| f_R \|_{\omega,1}$ since the best weighted $s$-term 
approximations to $f$ and $f_0$ are the same.  Theorem \ref{thm:infRIPprob} results then by application of Theorem \ref{thm:interpolation2} with $\gamma = N^{-\log^3(s)}$.

\subsection{Interpolation estimates in weighted $\ell_p$ spaces with $p \leq 1$}
A weakness of Theorem \ref{thm:infRIPprob} is that for a random draw of the sampling points, it gives guarantees with high probability only for a fixed function.  In order to derive uniform recovery bounds, or guarantees \emph{for all functions} in a given class for a {single} set of measurements, as opposed to guarantees for a particular function, we need to introduce a positive weight vector $v$ which dominates the weight vector $\omega$ in a suitable way. In order to illustrate the idea we start by recalling the error bound
\begin{equation}\label{error:bound:recall}
\| f - f^\sharp\|_{L_\infty} \leq c_1 \sigma_s(f)_{\omega,1} + d_1 \eta \sqrt{\frac{s}{m}},
\end{equation}
valid in the finite-dimensional setting:  $f = \sum_{j \in \Lambda} x_j \psi_j,$ $| \Lambda | = N \in \N$, and the samples are perturbed, $\sum_{\ell=1}^m |y_\ell - f(t_\ell)|^2 \leq \eta^2$.  In the infinite-dimensional setting, where $f = \sum_{j \in \Lambda} x_j \psi_j$ but $ \Lambda$ possibly countably infinite, we will treat the samples $y_\ell = f(t_\ell)$ as perturbed samples of a finite-dimensional approximation $f_0 = \sum_{j \in \Lambda_0} x_j \psi_j$ as in the probabilistic error analysis in the proof of Theorem \ref{thm:infRIPprob}, 
for some suitable $\Lambda_0 \subset \Lambda$. For a parameter $\alpha > 0$, the approximation error can then be bounded using
\begin{equation}\label{bound:finite:f0}
\| f - f_0 \|_{L_\infty} \leq \sum_{ j \notin \Lambda_0} |x_j| \|\psi_j\|_\infty \leq \max_{j \notin \Lambda_0}\{ \|\psi_j\|_\infty v_j^{-\alpha} \} 
\sum_{j \notin \Lambda_0} |x_j| v_j^\alpha \leq \max_{j \notin \Lambda_0} \{ w_j v_j^{-\alpha} \} \triple f \triple_{v^{\alpha},1}.
\end{equation}

On the right hand side, we obtain the norm $\triple f \triple_{v^{\alpha},1}$.
Recall, however, that for our compressive sensing approximation we can impose $\triple f \triple_{v,p}$ to be small for a small value of $p < 1$.  The following estimate will be useful for comparing weighted $p$ and $1$-norms.
\begin{lemma}\label{lem:boundp1} For a weight $\w$ and $0<p<1$, set $\alpha = \frac{2}{p} - 1$. Then $\| \x \|_{\omega^{\alpha},1} \leq \|\x\|_{\omega,p}.$
\end{lemma}
\begin{proof}
First observe that
\[
\left(\max_{j \in \Lambda_0} |x_j| \omega_j^{2/p-1}\right)^p = \max_{j \in \Lambda_0} |x_j|^p \omega_j^{2-p} \leq \sum_{j \in \Lambda_0} |x_j|^p \omega_j^{2-p}
= \|\x\|_{\omega,p}^p.
\]
The claimed inequality follows then from
\begin{align}
\|\x\|_{\omega^{\alpha},1} & = \sum_{j \in \Lambda_0} |x_j| \omega_j^\alpha 
\leq \left( \max_{j \in \Lambda_0} |x_j|^{1-p} \omega_j^{\alpha-2+p} \right) \sum_{j \in \Lambda_0} |x_j|^p \omega_j^{2-p}\notag\\
& \leq \left(\max_j |x_j| \omega_j^{2/p-1} \right)^{1-p} \|\x\|_{\omega,p}^p
\leq \|\x\|_{\omega,p}^{1-p} \|\x\|_{\omega,p}^p = \|\x\|_{\omega,p}.\notag
\end{align}
\end{proof}
Assuming that $v \geq \omega$ and $s \geq 2 \max_{j \in \Lambda_0} w_j^2$, say, the first term in the finite-dimensional error bound \eqref{error:bound:recall} with $f$ replaced by $f_0$ 
can be estimated using the Stechkin-type estimate of Theorem~\ref{thm:weighted:Stechkin} by
\begin{equation}\label{dd}
 \sigma_s(f_0)_{\omega,1} \leq c s^{1-1/p} \triple f_0 \triple_{\omega,p} \leq c s^{1-1/p} \triple f \triple_{v,p}.
\end{equation}
We aim to provide a bound of the second term on the right-hand side of \eqref{error:bound:recall} of the same order. Now $\eta := \sqrt{\sum_{\ell=1}^m |f_0(t_\ell) - f(t_\ell)|^2}
\leq \sqrt{m} \|f- f_0\|_{L_\infty},$ so by \eqref{bound:finite:f0} and Lemma \ref{lem:boundp1}, 
\[
\sqrt{\frac{s}{m}} \eta \leq \sqrt{s}  \max_{j \notin \Lambda_0} \{ w_j v_j^{-\alpha} \} \triple f \triple_{v^\alpha,1} \leq   \sqrt{s}  \max_{j \notin \Lambda_0} \{ w_j v_j^{-\alpha} \} \triple f \triple_{v,p},
\]
where we have applied Lemma~\ref{lem:boundp1} with $\alpha = \frac{2}{p}-1$ in the last step. The choice $\Lambda_0 = \Lambda_0^{(s,p)}$ with
\begin{equation}\label{def:Gammasp}
\Lambda_0^{(s,p)} := \{ j \in \Lambda: \omega_j v_j^{1-2/p} \geq s^{1/2-1/p} \} 
\end{equation}
therefore gives
\[
\eta \leq \sqrt{m} s^{1/2-1/p}  \triple f \triple_{v,p} \quad \mbox{so that} \quad \sqrt{\frac{s}{m}} \eta \leq s^{1-1/p} \triple f \triple_{v,p},
\]
and we have balanced the two error terms in \eqref{error:bound:recall} {after applying \eqref{dd}}. We still need 
to choose the weight $v$ so that $\Lambda_0^{s,p}$ is a finite set (ideally with size polynomial in $s$) 
and such that
the technical assumption $\max_{j \in \Lambda_0^{(s,p)}} w_j^2 \leq s/2$ is satisfied. 
The finiteness of $\Lambda_0^{(s,p)}$ is ensured when $(\omega_j v_j^{1-2/p})_{j \in \Lambda}$ 
is a sequence which converges to $0$ as $|j| \to \infty$. Moreover, if $v$ satisfies
\begin{equation}\label{cond:vj}
v_j^{2/p-1} \geq 2^{1/p-1/2} \omega_j^{2/p} = 2^{1/p-1/2} \omega_j \cdot \omega_j^{2(1/p-1/2)},
\end{equation}
then we have for all $j$ satisfying $\omega_j^2 \geq s/2$ that
\[
v_j^{2/p-1} \geq 2^{1/p-1/2} \omega_j (s/2)^{1/p-1/2} = \omega_j s^{1/p-1/2}.  
\]
In light of \eqref{def:Gammasp}, all $j \in \Lambda_0^{(s,p)}$ satisfy $\max_{j \in \Lambda_0^{(s,p)}} \omega_j^2 \leq s/2$.
{Inequality} \eqref{cond:vj} is satisfied if $v_j \geq 2 \omega_j^{1/(1-p/2)}.$
In particular, the choice
\[
v_j = 2 \omega_j^2
\]
is valid for all values of $p \in (0,1]$.  In this case $(\omega_j v_j^{1-2/p})_{j \in \Lambda}$ converges to $0$ as $|j| \rightarrow \infty$ if and only if $(\omega_j^{-1})_{j \in \Lambda}$ converges to $0$ as $| j | \rightarrow \infty$.  We can now state the main result.

\begin{theorem} 
\label{thm:intro2}
Let $p \in (0,1]$ and let $\omega,v$ be weights satisfying $\omega_j \geq \|\psi_j\|_\infty$ and $v_j \geq 2 \omega_j^{1/(1-p/2)}$.
For given $s \in \N$, let $\Lambda_0^{(s,p)} = \Lambda_0$ of size $N^{(s,p)}$ be the set of indices
$$
\Lambda_0 := \{ j \in \Lambda: \omega_j v_j^{1-2/p} \geq s^{1/2-1/p} \} .
$$
Fix a number $m$ of samples
\begin{equation}\label{bound:samples:inf}
m 
\geq c_0 s \max\{\log^3(s) \log(N^{(s,p)}), \log(1/\gamma)\}.
\end{equation}

\noindent
Suppose that the sampling points $t_\ell$, $\ell=1,\hdots,m,$ are drawn independently at random according to the orthogonalization measure for $( \psi_j )$. Then with probability exceeding $1-\gamma$ the following holds for all $f$ with 
$\triple f \triple_{v,p} < \infty$. 

Let $y_\ell = f(t_\ell)$, $\ell =  1,\hdots,m,$ and $\A$ be the $m \times N^{(s,p)}$ sampling matrix with entries $A_{j,\ell} = \psi_j(t_\ell)$, $j \in \Lambda_0$.
For $\tau \geq 1$, 
let $\x^\sharp$ be the solution to
\begin{equation}\label{l1:infinite}
\min \|\z\|_{\omega,1} \mbox{ subject to } \|\A \z - \y\|_2 \leq \tau s^{1/2-1/p} \sqrt{m} \triple f \triple_{v,p}
\end{equation}
and set $f^\sharp = \sum_{j \in \Lambda_0} x_j^\sharp \psi_j$.
Then $\|f - f^\sharp \|_\infty \leq C_\tau s^{1-1/p} \triple f \triple_{v,p}.$
\end{theorem}

\begin{proof}
Consider $f = \sum_{j \in \Lambda} x_j \psi_j$, and associated $f_0 = \sum_{j \in \Lambda_0^{(s,p)}} x_j \psi_j$.  We have
\[
\|f - f^\sharp\|_{L_\infty} \leq \|f-f_0\|_{L_\infty} + \|f_0 - f^\sharp\|_{L_\infty}.
\]
Since with probability at least $1-\gamma$ under the stated assumption on $m$, the matrix $\A$ has the $\w$-RIP and
thereby the weighted null space property of order $s$, we have
\[
\|f_0 - f^\sharp\|_{L_\infty} \leq C \tau s^{1-1/p} \triple f \triple_{v,p}
\]
by the observations preceding the statement of the theorem. Furthermore,
\[
\|f-f_0\|_\infty \leq s^{1/2-1/p}  \triple f \triple_{v,p} \leq s^{1-1/p} \triple f \triple_{v,p}
\]
by \eqref{bound:finite:f0}, Lemma~\ref{lem:boundp1}, and the definition of $\Lambda_0^{(s,p)}$. This concludes the proof with $C_\tau = C \tau + 1$.
\end{proof}

\section*{Acknowledgments}
Holger Rauhut acknowledges funding by the European Research Council through the grant StG 258926 and support by the Hausdorff Center for Mathematics at the University of Bonn. Rachel Ward was supported in part by an Alfred P. Sloan Research Fellowship, a Donald D. Harrington Faculty Fellowship, ONR Grant N00014-12-1-0743, NSF CAREER Award, and an AFOSR Young Investigator Program Award. Both Holger Rauhut and Rachel Ward would like to thank the Institute for Mathematics and Its Applications for its hospitality during a stay where this work was initiated.
\bibliography{weightedR}
\bibliographystyle{abbrv}

%


\end{document}